\documentclass[11pt,reqno]{amsart}
\usepackage{graphicx}
\usepackage[english]{babel}
\usepackage{amssymb,enumerate,bbm,amsmath}
\usepackage{url}
\usepackage[linktocpage=true,colorlinks=true,linkcolor=blue,citecolor=magenta,urlcolor=blue]{hyperref}
\usepackage{pgfplots}
\usepackage{tikz}
\usetikzlibrary{arrows,shapes,trees,backgrounds}
\usepackage[margin=1.1in]{geometry}
\newcommand{\ch}{\cosh}
\newcommand{\RR}{\mathbb{R}}
\newcommand{\HH}{\mathbb{H}}

\renewcommand{\P}{\partial}

\newcommand{\F}[2]{\frac{#1}{#2}}

\newtheorem{prop}{Proposition}[section]
\newtheorem{thm}[prop]{Theorem}
\newtheorem{lem}[prop]{Lemma}
\newtheorem{cor}[prop]{Corollary}

\newtheorem{assump}[prop]{Assumption}

\newtheorem*{claim}{Claim}
\theoremstyle{remark}
\newtheorem{rem}[prop]{Remark}

\numberwithin{equation}{section}
\allowdisplaybreaks
\begin{document}
\title[Shifted inverse curvature flows in hyperbolic space]{Shifted inverse curvature flows in hyperbolic space}

\author[X. Wang]{Xianfeng Wang}
\address{School of Mathematical Sciences and LPMC, Nankai University, Tianjin, 300071, P. R. China}
\email{\href{mailto:wangxianfeng@nankai.edu.cn}{wangxianfeng@nankai.edu.cn}}
\author[Y. Wei]{Yong Wei}
\author[T. Zhou]{Tailong Zhou}
\address{School of Mathematical Sciences, University of Science and Technology of China, Hefei, 230026, P.R. China}
\email{\href{mailto:yongwei@ustc.edu.cn}{yongwei@ustc.edu.cn}}
\email{\href{mailto:ztl20@ustc.edu.cn}{ztl20@ustc.edu.cn}}


\begin{abstract}
We introduce the shifted inverse curvature flow in hyperbolic space. This is a family of hypersurfaces in hyperbolic space expanding by $F^{-p}$ with positive power $p$ for a smooth, symmetric, strictly increasing and $1$-homogeneous curvature function $F$ of the shifted principal curvatures with some concavity properties. We study the maximal existence and asymptotical behavior of the flow for horo-convex hypersurfaces. In particular, for $0<p\leq 1$ we show that the limiting shape of the solution is always round as the maximal existence time is approached. This is in contrast to the asymptotical behavior of the (non-shifted) inverse curvature flow, as Hung and Wang \cite{HW2015} constructed a counterexample to show that the limiting shape of inverse curvature flow in hyperbolic space is not necessarily round.
\end{abstract}

\maketitle
\tableofcontents

\section{Introduction}\label{sec:1}
Let $\Sigma_0$ be a smooth closed hypersurface in hyperbolic space $\mathbb{H}^{n+1}$ parameterized by the embedding $X_0:\Sigma \to X_0(\Sigma)=\Sigma_0\subset \mathbb{H}^{n+1}$ with $n\geq2$. We consider the following shifted inverse curvature flow in $\mathbb{H}^{n+1}$, which is  a family of embeddings $X:\Sigma\times[0,T^*)\hookrightarrow\mathbb{H}^{n+1}$ satisfying
\begin{equation}\label{eq-flow}
\left\{\begin{aligned}
\frac{\partial}{\partial t}X(x,t)=&~\frac 1{F^p(x,t)}{\nu}(x,t),\\
X(\cdot,0)=&~X_{0},
\end{aligned}\right.
\end{equation}
where $p>0$, ${\nu}$ is the outward unit  normal of  the evolving hypersurface $\Sigma_{t}=X(\Sigma,t)$, and $F(x,t)=f(\kappa(x,t))$ is a smooth, symmetric function of the shifted principal curvatures
\begin{equation*}
  (\kappa_1,\cdots,\kappa_n)=(\lambda_1-1,\cdots,\lambda_n-1).
\end{equation*}
Here $\lambda=(\lambda_1,\cdots,\lambda_n)$ are the principal curvatures of $\Sigma_t$ which are defined by eigenvalues of the Weingarten matrix $\mathcal{W}=(h_i^j)$. Therefore, the shifted principal curvatures are eigenvalues of the shifted Weingarten matrix $\mathcal{W}-I$.

A motivation to study the flow \eqref{eq-flow} with speed function depending on the shifted  principal curvatures is the notion of horo-convexity of a hypersurface in  hyperbolic space. Recall that a smooth hypersurface $\Sigma$ in  hyperbolic space $\mathbb{H}^{n+1}$ is called to be horo-convex if it is convex by horospheres, and equivalently if all of its principal curvatures are greater than $1$. A horo-convex hypersurface can be reparametrized by the horospherical support function $s: \mathbb{S}^n\to \mathbb{R}_+$ as described by the second author with Andrews and Chen \cite[\S 5]{ACW18}, and the horo-convexity is characterized in terms of the positivity of a matrix $A_{ij}[s]$ (see \eqref{s6:2-sA} for the definition) on the sphere $\mathbb{S}^n$ involving up to second derivatives of $s$:
\begin{equation}\label{s1:s1}
  \left(h_i^k-\delta_i^k\right)A_{kj}[s]=e^{-s}\sigma_{ij},
\end{equation}
where $\sigma_{ij}$ denotes the canonical metric on the sphere. The identity \eqref{s1:s1} motivates the definition of hyperbolic principal radii of curvature $1/{|\lambda_i-1|}$, which would play a similar role in $\mathbb{H}^{n+1}$ of the Euclidean principal radii of curvature. A similar development related to the hyperbolic principal radii of curvature $1/{|\lambda_i-1|}$ was presented earlier by Espinar, G\'{a}lvez and Mira in \cite{EGM09}.  See \S \ref{sec:2-2} for more introduction on horo-convex hypersurfaces and the horospherical support function.

Denote by $\Gamma_+=\{(\kappa_1,\cdots,\kappa_n)\in \mathbb{R}^n:~\kappa_i>0\}$ the positive cone in $\mathbb{R}^n$. In this paper, we always assume that $f$ satisfies the following assumption.
\begin{assump}\label{s1:assum1}
$f$ is a smooth, symmetric, strictly increasing and 1-homogeneous function on $\mathbb{R}^n$ satisfying $f>0$ in $\Gamma_+$ and $f(1,\dots,1)=n$.
\end{assump}

\subsection{Main results}$\ $

We first prove the following result.
\begin{thm}\label{thm-0<p<infty}
Assume that $f$ is a concave function satisfying Assumption \ref{s1:assum1} and $f|_{\P\Gamma_+}=0$. Given a smooth, closed horo-convex hypersurface $\Sigma_0$ in hyperbolic space, for any $0<p<\infty$, we have
\begin{itemize}
  \item[(i)] The flow \eqref{eq-flow} has a smooth solution $\Sigma_t=X(\Sigma,t)$ which is defined on a maximal finite interval $[0,T^*)$ and is horo-convex for any $t\in [0,T^*)$.
  \item[(ii)] We introduce geodesic polar coordinate system with center in the domain enclosed by $\Sigma_0$. The leaves $\Sigma_{t}$ can be written as graphs of a function $u=u(\xi,t)$ over $\mathbb{S}^n$ and we have that
\begin{equation}
\limsup_{t\rightarrow T^*}\max_{\mathbb{S}^n}u(\cdot,t)=\infty.
\end{equation}
Let $\theta(t)=\theta(t,\theta_0)=\theta(t,T^*)$ be the radius of the spherical solution to flow \eqref{eq-flow} with initial hypersurface given by a geodesic sphere of radius $\theta_0$ and with the same maximal existence time $T^*$. For any $m\in \mathbb{N}$, we have $ |u-\theta|_{C^m(\mathbb{S}^n)}\leq c_m$ for positive constants $c_m$ depending on $p,\Sigma_0$, and thus the rescaled graph funtion $u\theta^{-1}$ converges to 1 smoothly.
  \item[(iii)] The leaves $\Sigma_t$ converge umbilically in the sense
\begin{equation}\label{s1:thm1-1}
c_1Q(t)^{-1}\leq\lambda_i-1\leq c_2Q(t)^{-1},\quad \forall~t\in [0,T^*),
\end{equation}
for all $i=1,\cdots,n$, where $c_1,\ c_2$ are some positive constants depending only on $p$ and $\Sigma_0$, $\lambda_i$ is the $i$th principal curvature of $\Sigma_t$,  and $Q(t)$ is defined by
\begin{equation}\label{s1:Q}
  Q(t)=\F {e^{2\theta(t)}-1}2
\end{equation}
which is the reciprocal of the shifted principal curvature of the geodesic sphere of radius $\theta(t,T^*)$.
\end{itemize}
\end{thm}

Both the radius $\theta(t,T^*)$ of the spherical solution to \eqref{eq-flow} and $Q(t)$ defined in \eqref{s1:Q} converge to infinity as $t\to T^*$, so the estimate \eqref{s1:thm1-1} means that all of the principal curvatures of $\Sigma_t$ converge to $1$ as $t\to T^*$. If we define a new time parameter $\tau=\tau(t)$ by the relation
\begin{equation*}
\F{d\tau}{dt}=Q(t)^p
\end{equation*}
with $\tau(0)=0$, by the flow equation \eqref{eq-flow} and the definition \eqref{s1:Q} of $Q(t)$, we have
\begin{equation*}
  \frac d{d\tau}\theta(t(\tau))=\frac d{dt}\theta\cdot \frac{dt}{d\tau}=\frac 1{n^{p}}.
\end{equation*}
This implies that
\begin{equation}\label{s1:thet-tau}
  \theta(\tau)-\theta_0=n^{-p}\tau,
\end{equation}
and $\tau$ ranges from $0$ to $+\infty$. Then the estimate \eqref{s1:thm1-1} is equivalent to that
\begin{equation*}
  C^{-1}e^{-\frac 2{n^{p}}\tau}\leq \lambda_i-1\leq Ce^{-\frac 2{n^{p}}\tau},
\end{equation*}
which means that the shifted principal curvatures convergence to zero exponentially.

We compare the results in Theorem \ref{thm-0<p<infty} for flow \eqref{eq-flow} with those for (non-shifted) inverse curvature flow
\begin{equation}\label{s1:ICF}
  \frac{\partial}{\partial t}X(x,t)=~\frac 1{F^p(x,t)}{\nu}(x,t)
\end{equation}
in hyperbolic space, where $F(x,t)=f(\lambda(x,t))$ is a function of the principal curvatures $\lambda=(\lambda_1,\cdots,\lambda_n)$ satisfying Assumption \ref{s1:assum1}. Assume further that $f$ is concave and $f|_{\partial\Gamma_+}=0$, the flow \eqref{s1:ICF} for hypersurfaces in hyperbolic space has been studied by Gerhardt \cite{Ge-2011} (for $p=1$) and Scheuer \cite{Julian} (for $p>0$). However, for $p>1$ the convergence of flow \eqref{s1:ICF} can only be proved under strong assumptions that either the oscillation of the initial hypersurface is sufficiently small \cite{Julian} or the curvature of the initial hypersurface is sufficiently pinched \cite{K-J}, except that in two dimensional case the convergence of flow \eqref{s1:ICF} with $F=\sqrt{K}$ (the square root of the Gauss curvature) has been proved by Scheuer \cite{Sche15grad}  for $0<p\leq2$ and by the third author and Li \cite{LZ19} for all $p>0$ under an extra assumption that the initial surface has positive intrinsic scalar curvature. Our Theorem \ref{thm-0<p<infty} says that the convergence of the shifted inverse curvature flow \eqref{eq-flow} in  hyperbolic space can be proved for all $p>0$.

The inverse curvature flow \eqref{s1:ICF} in Euclidean space has already been studied extensively, see for instance \cite{Gerh90,Gerh14,Q-Li10,L-W-W,Schn06,Urb90MZ,Urb91JDG}. It can be proved that the limiting shape of the inverse curvature flow in Euclidean space is always round.  If we write the solution $\Sigma_t$ of flow \eqref{s1:ICF} in Euclidean space as graphs of radial function $u(\cdot,t)$ on $\mathbb{S}^n$ with respect to some origin $o$, and let $\theta(t)$ be the corresponding spherical solution with initial sphere intersecting the initial hypersurface $\Sigma_0$ or with the same maximal existence time $T^*$ (if $T^*$ is finite). Then it is always true that $u/\theta$ converges to the constant $1$ as the maximal existence time is approached. Since the induced metric on $\Sigma_t=\mathrm{graph} ~u(\cdot,t)$ is given locally by $g_{ij}=u_iu_j+u^2g_{\mathbb{S}^n}$, this implies that the rescaled metric $\theta(t)^{-2}g_{ij}(t)$ converges to a round metric on the sphere. Moreover, an improved asymptotical roundness of the inverse curvature flow in Euclidean space was obtained by Schn\"{u}rer \cite{Schn06} and Scheuer \cite{Julian2}, where they proved the existence of an optimal center $o$ such that the oscillation of the graphical representation $u(\cdot,t)$ of $\Sigma_t$ with respect to $o$ converges to zero exponentially and the flow becomes arbitrarily close to a flow of spheres.

However, the above mentioned asymptotical roundness is not always true for the solution to the inverse curvature flow in hyperbolic space. Though the asymptotical umbilic behavior of the solution $\Sigma_t$ as $t\to T^*$ can be proved, this doesn't imply that the limiting shape of $\Sigma_t$ is round in the sense that the circumradius minus inradius of $\Sigma_t$ may not decay to zero as $t\to T^*$ and the rescaled metric $e^{-2\theta(t)}g(t)$ of the induced metric on $\Sigma_t$ does not converge to a round metric. In fact, it is already demonstrated by Neves \cite{Neves10} that the limiting shape of the inverse mean curvature flow in an asymptotically hyperbolic space is not necessarily round, and therefore it is impossible to prove the hyperbolic Penrose inequality by the method of inverse mean curvature flow. Later Hung and Wang \cite{HW2015} constructed an example to show the limiting shape of the inverse mean curvature flow in hyperbolic space is not necessarily round as well.
In a joint work with Li \cite{L-W-W}, the first two authors showed that the construction in \cite{HW2015} also works for inverse powers of mean curvature flow with power $0<p<1$.

In the second part of this paper, we prove the asymptotical roundness of our shifted inverse curvature flow \eqref{eq-flow} for $0<p\leq 1$. This shows that our shifted inverse curvature flow behaves better at  infinity than  the non-shifted inverse curvature flow \eqref{s1:ICF}. More precisely, we have
\begin{thm}\label{thm-0<p<1}
Let $f$ be a function satisfying Assumption \ref{s1:assum1}. Assume further that either
\begin{itemize}
  \item[(a)] $f$ is concave and $f|_{\P\Gamma^+}=0$; or
  \item[(b)] $f$ is concave and $f$ is inverse concave, i.e.,
  \begin{equation}\label{s1:f*}
    f_*(\kappa_1,\cdots,\kappa_n):=\frac 1{f(\frac 1{\kappa_1},\cdots,\frac 1{\kappa_n})}
  \end{equation}
  is concave; or
  \item[(c)] $f$ is inverse concave and $f_*|_{\P\Gamma^+}=0$; or
  \item[(d)] $n=2$.
\end{itemize}
Given a smooth, closed horo-convex hypersurface $\Sigma_{0}$ in $\mathbb{H}^{n+1}$. For any $0<p\leq 1$, the flow \eqref{eq-flow} has a smooth solution $\Sigma_t=X(\Sigma,t)$ which is defined on a maximal finite interval $[0,T^*)$. The solution $\Sigma_t$ preserves the horo-convexity, and converges smoothly and umbilically as in Theorem \ref{thm-0<p<infty}.

Moreover, there exists a point $y\in\HH^{n+1}$ such that for any $\beta<(1+\frac{2}{n})p$ we have
\begin{equation*}
  |u_y(\cdot,t)-\theta(t,T^*)|\leq C Q(t)^{-\beta}
\end{equation*}
for some positive constant $C=C(p,\beta,\Sigma_0)$, where $u_y(\cdot,t)$ is the graphical representation of $\Sigma_t$ in the geodesic polar coordinate centered at $y$. Equivalently, let $S_t$ be the spherical solution  of \eqref{eq-flow} given by the family of geodesic spheres of radius $\theta(t,T^*)$ and centered at $y$. Then for any $\beta<(1+\frac{2}{n})p$, we have
\begin{equation}\label{s1:2-1}
d_{\mathcal{H}}\left(\Sigma_{t},S_t\right)\leq CQ(t)^{-\beta}
\end{equation}
for some positive constant $C=C(p,\beta,\Sigma_0)$, where $d_{\mathcal{H}}$ is the hyperbolic Hausdorff distance of compact sets and $Q(t)$ is defined in \eqref{s1:Q}.
\end{thm}
\begin{rem}
By introducing the new time parameter $\tau$ as in \eqref{s1:thet-tau}, the estimate \eqref{s1:2-1} is equivalent to
\begin{equation}\label{s1:2-2}
d_{\mathcal{H}}\left(\Sigma_{\tau},S_{\tau}\right)\leq Ce^{-\frac 2{n^p}\beta\tau}
\end{equation}
for any $\beta<(1+\frac{2}{n})p$. Thus we have the exponential convergence of the solution to the spherical solution as $\tau$ goes to infinity.
\end{rem}
In the end of this paper, we show by a concrete example that there exists a strictly horo-convex initial hypersurface which develops a principal curvature less than 1 quickly.  Precisely, we obtain that for powers $p>1$  of the shifted inverse  mean curvature flow, i.e., $f(\kappa)=\sum_{i=1}^n\kappa_i$, the horo-convexity might be lost. This indicates  that the condition $f|_{\P\Gamma^+}=0$ in Theorem \ref{thm-0<p<infty} (when $p>1$) is required, and that the results in Theorem \ref{thm-0<p<1} (cases (b) (c) (d)) cannot be generalized to the case $p>1$.

\subsection{Organization of the paper}$\ $

This paper is organized as follows. In Section \ref{sec:2}, we collect some properties of smooth symmetric functions and some preliminaries on horo-convex hypersurfaces in hyperbolic space. In Section \ref{sec:evl}, we derive several evolution equations along the shifted inverse curvature flow \eqref{eq-flow}.

In Section \ref{sec:4}, we obtain a priori $C^0$ and $C^1$ estimates for smooth horo-convex solution to flow \eqref{eq-flow}. We first apply Alexandrov-reflection argument and a basic property of horo-convex hypersurfaces to obtain the $C^0$ estimate for smooth horo-convex solution $\Sigma_t$ of flow \eqref{eq-flow}. In particular, we show that the oscillation of $u(\cdot,t)$ is uniformly bounded. This together with an estimate in \cite[Theorem 2.7.10]{Ge-2006} implies a uniform $C^1$ estimate on $\Sigma_t$. However, this is not sufficient for our purpose to study the asymptotical behavior of the flow. In stead, we explore the horo-convexity of $\Sigma_t$ and refine the argument in \cite[Theorem 2.7.10]{Ge-2006} to obtain an improved $C^1$ estimate.

In Section \ref{sec:thm1}, we prove Theorem \ref{thm-0<p<infty}. We first show that the solution $\Sigma_t$ expands to  infinity as $t\to T^*$. To study the asymptotical behavior of the solution, we consider the rescaled curvature quantities $fQ$ and $\kappa_iQ$. We prove uniform positive two-sides bounds on these quantities and derive the estimate \eqref{s1:thm1-1}. We then apply parabolic regularity theory to show the uniform higher regularity of the shifted defining function $u(\cdot,t)-\theta(t)$ of $\Sigma_t$.

Theorem \ref{thm-0<p<1} is proved in Section \ref{sec:pinc} and Section \ref{sec:osc}. We first establish the pinching estimates for the shifted principal curvatures of the solution in Section \ref{sec:pinc}. Precisely, we show that the ratio of the largest shifted principal curvature $\kappa_n$ to the smallest one $\kappa_1$ converges to $1$ exponentially as $t\to T^*$. The key tool is Andrews' refined tensor maximum principle \cite{BAndrews07}. We also derive the regularity estimate of the evolving hypersurfaces during their expansion to  infinity. In Section \ref{sec:osc} we first use the pinching estimates and the conformally flat parametrization to show the Hausdorff convergence of $\Sigma_t$ to a sphere, i.e., the outer-radius minus inner-radius of $\Sigma_t$ decays to zero exponentially. Combining this with the evolution of $u(\cdot,t)-\theta(t)$, we prove the existence of optimal oscillation minimizing center $y\in \mathbb{H}^{n+1}$ with the property that $\mathrm{osc} (u_y(\cdot,t))\leq CQ(t)^{-\epsilon}$ for some positive constant $\epsilon$. Finally, we apply the linearization of the flow to improve the convergence rate and complete the proof of Theorem \ref{thm-0<p<1}.

In the last section, a counterexample is constructed to show the loss of horo-convexity for flow \eqref{eq-flow} with power $p>1$ and $f$ given by the shifted mean curvature.


\section{Preliminaries}\label{sec:2}

In this section, we collect some properties of smooth symmetric curvature functions and some preliminaries on horo-convex hypersurfaces in hyperbolic space.

\subsection{Symmetric functions}\label{sec:2-3} $\ $

Given a smooth symmetric function $f$ on the positive cone $\Gamma_+\subset \mathbb{R}^n$, a result of Glaeser \cite{Gla} implies that there is a smooth $O(n)$-invariant function $F$ on the subspace of $\mathrm{Sym}(n)$ of symmetric positive definite matrices such that $f(\kappa(A))=F(A)$, where $\kappa(A)=(\kappa_1,\cdots,\kappa_n)$ are the eigenvalues of $A$. We denote by $\dot{F}^{ij}$ and $\ddot{F}^{ij,kl}$ the first and second derivatives of $F$ with respect to the components of its argument.  We also use the notations $\dot{f}^i(\kappa)$, $\ddot{f}^{ij}(\kappa)$ to denote the derivatives of $f$ with respect to $\kappa$. At any diagonal $A$, we have
\begin{equation*}
  \dot{F}^{ij}(A)=\dot{f}^i(\kappa(A))\delta_i^j.
\end{equation*}
If the eigenvalues of $A$ are mutually different, the second derivative $\ddot{F}$ of $F$ in direction $B\in \mathrm{Sym}(n)$ is given in terms of $\dot{f}$ and $\ddot{f}$ by  (see e.g., \cite{BAndrews07}):
\begin{equation}\label{s2:F-ddt}
  \ddot{F}^{ij,kl}(A)B_{ij}B_{kl}=\sum_{i,k}\ddot{f}^{ik}(\kappa(A))B_{ii}B_{kk}+2\sum_{i>k}\frac{\dot{f}^i(\kappa(A))-\dot{f}^k(\kappa(A))}{\kappa_i(A)-\kappa_k(A)}B_{ik}^2.
\end{equation}
This formula makes sense as a limit in the case of any repeated values of $\kappa_i$. We have the following properties for concave functions. See e.g., \cite[\S 2]{A-M-Z} and \cite[Lemma 2.2.20]{Ge-2006}.
\begin{lem}\label{s2:lem-conc}
A smooth symmetric function $F$ on $\mathrm{Sym}(n)$ is concave in $A$ if and only if $f$ is concave in $\kappa(A)$ and $(\dot{f}^i-\dot{f}^k)(\kappa_i-\kappa_k)\leq 0$ for any $i\neq k$. Moreover, if $f$ is concave, 1-homogeneous and is normalized to satisfy $f(1,\cdots,1)=n$, then
\begin{equation}\label{s2:conc}
  \sum_{i}\dot{f}^{i}\geq n,\quad \ f\leq \sum_i\kappa_i.
\end{equation}
\end{lem}
The examples of concave symmetric functions include: (i) $E_k^{1/k}$; (ii)  $(E_k/{E_l})^{1/{(k-l)}}$ with $k>l$, where
\begin{equation*}
  E_k=\binom{n}{k}^{-1}\sigma_k(\kappa)=\binom{n}{k}^{-1}\sum_{1\leq i_1<\cdots<i_k\leq n}\kappa_{i_1}\cdots\kappa_{i_k};
\end{equation*}
and (iii) the power means $  H_r=(\sum_{i=1}^n\kappa_i^r)^{1/r} $ with $r\leq 1$. Taking convex combinations or geometric means of the above concave examples can produce more concave examples.

For any positive definite symmetric matrix $A\in \mathrm{Sym}(n)$, we define $F_*(A)=F(A^{-1})^{-1}$. Then $F_*(A)=f_*(\kappa(A))$, where $f_*$ is the dual function of $f$ defined in \eqref{s1:f*}.
We say that a symmetric function $F$ is inverse-concave if $F_*(A)$ is concave. The following lemma characterizes the inverse concavity of $f$ and $F$ (see \cite{BAndrews07,A-M-Z,A-W}).
\begin{lem}\label{inv-concave}
\begin{itemize}
  \item[(i)] $F_*$ is concave on $\Gamma_+$ if and only if $f_*$ is concave on $\Gamma_+$.
   \item[(ii)]If $f$ is inverse concave, then
 \begin{equation}\label{s2:inv-conc-1}
   \sum_{k,l=1}^n\ddot{f}^{kl}y_ky_l+2\sum_{k=1}^n\frac {\dot{f}^k}{\kappa_k}y_k^2~\geq ~2f^{-1}(\sum_{k=1}^n\dot{f}^ky_k)^2
 \end{equation}
 for any $y=(y_1,\cdots,y_n)\in \mathbb{R}^n$, and
\begin{equation}\label{s2:inv-conc}
\frac{\dot{f}^k-\dot{f}^l}{\kappa_k-\kappa_l}+\frac{\dot{f}^k}{\kappa_l}+\frac{\dot{f}^l}{\kappa_k}\geq~0,\quad \left(\dot{f}^k\kappa_k^2-\dot{f}^l \kappa_{l}^2\right)(\kappa_k-\kappa_l)\geq 0,\quad \forall~k\neq l.
\end{equation}
 \item[(iii)] If $f$ is inverse concave, 1-homogeneous  and $f(1,\cdots,1)=n$, then
  \begin{equation}\label{s2:ic-2}
  \sum_{k=1}^n \dot{f}^k\kappa_k^2\geq~ f^2/n.
  \end{equation}
\end{itemize}
\end{lem}

The examples (i) $E_k^{1/k}$; (ii)  $(E_k/{E_l})^{1/{(k-l)}}$ with $k>l$; and (iii) the power means $H_r$ with $r\geq -1$ are smooth inverse-concave symmetric functions. Also, taking convex combinations or geometric means of the inverse-concave examples can produce more inverse-concave examples.

\subsection{Horo-convex hypersurface in $\mathbb{H}^{n+1}$}\label{sec:2-2} $\ $

Let $\Sigma$ be a smooth closed hypersurface in $\mathbb{H}^{n+1}$. We denote by $g_{ij}, h_{ij}$ and $\nu$ the induced metric, the second fundamental form and unit outward normal vector of $\Sigma$. The Weingarten map is denoted by $\mathcal{W}=(h_i^j)$, where $h_i^j=h_{ik}g^{kj}$. The principal curvatures $\lambda=(\lambda_1,\cdots,\lambda_n)$ of $\Sigma$ are defined by the eigenvalues of $\mathcal{W}$. We denote the shifted principal curvatures by
\begin{equation*}
  (\kappa_1,\cdots,\kappa_n)=(\lambda_1-1,\cdots,\lambda_n-1),
\end{equation*}
which are eigenvalues of the shifted Weingarten matrix $\mathcal{W}-I$. The hypersurface $\Sigma$ is called horo-convex if and only if it is convex by horospheres:
The horospheres in hyperbolic space are hypersurfaces with constant principal curvatures equal to 1 everywhere. In the Poincare disk model $\mathbb{B}^{n+1}$ of  hyperbolic space, the horospheres are just spheres touching at the boundary point of $\mathbb{B}^{n+1}$. Given any point $z\in S^n_{\infty}=\partial \mathbb{H}^{n+1}$, there is a family of horospheres touching at $z$ and foliating the whole space $\mathbb{H}^{n+1}$. A hypersurface $\Sigma$ is horo-convex if for any point $x\in\Sigma$ there exists a horosphere enclosing $\Sigma$ and touching $\Sigma$ at $x$. Therefore $\Sigma$ is horo-convex if its shifted principal curvatures satisfy $\kappa_i=\lambda_i-1>0$ for all $i=1,\cdots,n$.

Horo-convex hypersurfaces are important class of hypersurfaces in hyperbolic space. In \cite[\S 5]{ACW18}, the second author with Andrews and Chen developed the theory of horospherical support function. For a horo-convex hypersurface $\Sigma$ in $\mathbb{H}^{n+1}$, the horospherical support function $s:\mathbb{S}^n\to \mathbb{R}$ is defined such that for each $z\in \mathbb{S}^n$, $s(z)$ is the signed geodesic distance of the smallest horosphere that encloses $\Sigma$ and touches $S^n_{\infty}=\partial \mathbb{H}^{n+1}$ at $z$ to the origin of $\mathbb{H}^{n+1}$. We can use $s(z)$ to parametrize $\Sigma$ as the embedding $X: \mathbb{S}^n\to \mathbb{H}^{n+1}$ by
\begin{align}\label{s6:2-s}
  X(z)= & \biggl(-e^s\bar{\nabla}s+\left(\frac 12e^s|\bar{\nabla}s|^2-\sinh s\right)z, \frac 12e^s|\bar{\nabla}s|^2+\cosh s\biggr),
\end{align}
where the right hand side of \eqref{s6:2-s} denotes a point (for each $z$) in the hyperboloid $\mathbb{H}^{n+1}$ of the Minkowski space $\mathbb{R}^{n+1,1}$. Moreover, we derived the following identity to relate the shifted Weingarten matrix of $\Sigma$ to a positive matrix $A_{ij}$ on the sphere $\mathbb{S}^n$ (see \cite[\S 5.4]{ACW18})
\begin{equation}\label{s6:2-s1}
  \left(h_i^k-\delta_i^k\right)A_{kj}[s]=e^{-s}\sigma_{ij},
\end{equation}
where
\begin{equation}\label{s6:2-sA}
  A_{ij}[s]=\bar{\nabla}_i(e^s\bar{\nabla}_js)-\frac 12e^s|\bar{\nabla}s|^2\sigma_{ij}+\sinh s\sigma_{ij},
\end{equation}
and $\sigma_{ij}$ is the canonical metric on $\mathbb{S}^n$. Thus a smooth function $s:\mathbb{S}^n\to \mathbb{R}_+$ defines a horo-convex hypersurface in hyperbolic space if and only if the matrix $A_{ij}$ is positive definite. The identity \eqref{s6:2-s1} is also a motivation for us to consider the flows by functions of shifted principal curvatures. Moreover, we can rewrite the flow \eqref{eq-flow} as a scalar parabolic equation of the horospherical support function on the sphere $\mathbb{S}^n$:
\begin{equation}\label{s2:1-s2}
  \frac{\partial }{\partial t}s=e^{ps}F^p_*(A_i^j[s]),
\end{equation}
which implies the short-time existence of the flow \eqref{eq-flow}. The equation \eqref{s2:1-s2} will also be used to derive the $C^{2,\alpha}$ estimate of the solution $\Sigma_t$ in the case that $F$ is inverse concave and $0<p\leq 1$ in Section \ref{sec:pinc}.

A horo-convex hypersurface $\Sigma$ in hyperbolic space has a nice property that the outer-radius of the domain $\Omega$ enclosed by $\Sigma$ is controlled by its inner-radius. Recall that the inner-radius $\rho_-$ and outer-radius $\rho_+$ of a bounded domain $\Omega$ are defined by
\begin{equation*}
\rho_-= ~\sup\{\rho: B_{\rho}(y)\subset\Omega \text{\rm\ for some }y\in{\mathbb H}^{n+1}\}
\end{equation*}
and
\begin{equation*}
\rho_+= ~\inf\{\rho: \Omega\subset B_{\rho}(y)\text{\rm\ for some }y\in{\mathbb H}^{n+1}\},
\end{equation*}
where $B_{\rho}(y)$ denotes the geodesic ball of radius $\rho$ centered at $y$ in $\mathbb{H}^{n+1}$.
\begin{thm}[\cite{Bor-Miq}]
Let $\Omega$ be a compact horo-convex domain in $\mathbb{H}^{n+1}$ and denote the center of an inball by $o$ and its inner radius by $\rho_-$. Then the maximum of the distance $\mathrm{d}_{\mathbb{H}}(o,\cdot)$ between $o$ and the points on $\partial\Omega$ satisfies
  \begin{equation}\label{s2:inner}
  \max_{p\in\partial\Omega} \mathrm{d}_{\mathbb{H}}(o,p)~\leq~ \rho_-+\ln\frac{\big(1+\sqrt{\tanh(\rho_-/2)}\big)^2}{1+\tanh(\rho_-/2)}<~\rho_-+\ln 2.
\end{equation}
\end{thm}

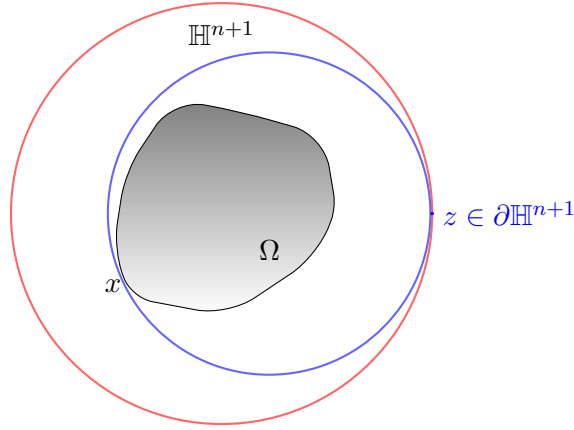
\begin{figure}
\begin{tikzpicture}[ scale=.8]
\draw[color=red!60, thick](-2,-2) circle (3.5);
\node at (-2,1) {$\mathbb{\mathbb{H}}^{n+1}$};

\path[shade,draw]
(-3.5,-3.4) [rounded corners=10pt] -- (-3.8,-2.6) -- (-3.6,-1.3)
 [rounded corners=12pt]--(-2.8,-0.1)
 [rounded corners=12pt]--(-1.7,-0.3)
 [rounded corners=12pt]--(-0.3,-0.7)
 [rounded corners=5pt]--(-0.1,-1.9)
  [rounded corners=12pt]--(-0.4,-2.7)--(-1.9,-3.7)
  [rounded corners=8pt]--cycle;
  \node at (-1.2,-2.6) {$\Omega$};

\draw[color=blue!60, thick](-1.19,-2) circle (2.67);
\draw[blue,thick] (1.5,-2) circle (0.01) node[right] {$z\in \partial \mathbb{H}^{n+1}$};\node at (-3.8,-3.2) {$x$};
\end{tikzpicture}
\caption{Horo-convex region $\Omega$ in Poinc\'{a}re disk ball}
\end{figure}

\subsection{Radial graphical representation}$\ $

Let $\Sigma\subset\HH^{n+1}$ be a horo-convex hypersurface. We can choose a point $o$ in the domain enclosed by $\Sigma$ such that $\Sigma$ is star-shaped with respect to $o$ and $\Sigma$ can be parametrized as a radial graph over $\mathbb{S}^{n}$. We fix $o\in\mathbb{H}^{n+1}$ and consider geodesic polar coordinates centered at $o$, then the hyperbolic space $\mathbb{H}^{n+1}$ can be regarded as a warped product space $[0,+\infty)\times\mathbb{S}^{n}$ with metric $$\bar{g}=dr^{2}+\sinh^{2}r\sigma_{ij}d\xi^{i}d\xi^{j},$$
where $\xi=\{\xi^{i}\}_{i=1}^{n}$ are local coordinates on $\mathbb{S}^{n}$ and $(\sigma_{ij})=(g_{\mathbb{S}^n}(\partial_{\xi^i},\partial_{\xi^j}))$ is the canonical metric of $\mathbb{S}^{n}$.  Then we can express $\Sigma$ in the following form:
$$\Sigma=\{(\xi,u(\xi))\mid\ u:\mathbb{S}^{n}\rightarrow\mathbb{R}^{+},\ \xi\in \mathbb{S}^{n}\},$$
where $u$ is a smooth function on $\mathbb{S}^n$.

We define $\varphi:\mathbb{S}^{n}\rightarrow\mathbb{R}$ by $\varphi(\xi)=\psi(u(\xi))$, where $\psi$ is a positive function which satisfies $\psi'(r)=1/{\sinh r}$. Let $\varphi_{i}=\bar{\nabla}_{i}\varphi$ and $\varphi_{ij}=\bar{\nabla}_{i}\bar{\nabla}_{j}\varphi$ denote the covariant derivatives of $\varphi$ with respect to the metric $\sigma$ on $\mathbb{S}^{n}$.  Then the tangential vectors on $\Sigma$ are given by
$$X_{i}=\partial_{i}+u_{i}\partial_{r}=\partial_{i}+\sinh u\varphi_{i}\partial_{r},$$
and  we can express the induced metric $g_{ij}$ of the graph $\Sigma=\mathrm{graph} ~u$ as
\begin{equation}\label{g_ij}
g_{ij}=\sinh^{2}u(\sigma_{ij}+\varphi_{i}\varphi_{j}).
\end{equation}
Denote
\begin{equation}\label{s2:v-def}
v=\sqrt{1+|\bar{\nabla}\varphi|^{2}_{\sigma}}.
\end{equation}
Then the unit outward normal has the form
\begin{equation*}
{\nu}=\frac1{v}(\partial_{r}-\frac{u^{i}}{\sinh^2u}\partial_{i}),
\end{equation*}
where $u^{i}=u_{j}\sigma^{ij}$. It follows that the second fundamental form $(h_{ij})$ on $\Sigma$ in the coordinates $(\xi^i)$ is given by
\begin{align}
h_{ij}=&\frac{\coth u}{v}g_{ij}-\frac{\sinh u}v\varphi_{ij},\label{s2:hij}
\end{align}
and we have the Weingarten matrix
\begin{align}
h_{i}^{j}=g^{jk}h_{ki}=&\frac{\coth u}{v}\delta_{i}^{j}-\frac{\sinh u}{v}\varphi_{ik}g^{kj},\label{s2:h-hat0}\\
=&\frac{\coth u}{v}\delta_{i}^{j}+\frac{\ch u}{v^3\sinh^3u}u_iu^j-v^{-1}g^{jk}\bar{\nabla}_i\bar{\nabla}_ku,\label{s2:h-hat}
\end{align}
where $(g^{ij})=(g_{ij})^{-1}$ is given by
\begin{equation}\label{s2:g-ni}
 g^{ij}=\frac 1{\sinh^2 u}\left(\sigma^{ij}-\frac{\varphi^{i}\varphi^{j}}{v^{2}}\right).
\end{equation}
To obtain the equality in \eqref{s2:h-hat}, we used the relation
\begin{align}\label{s2:uk}
  g^{jk}u_k=&\frac 1{\sinh^2u}\left(\sigma^{jk}-\frac{u^{j}u^{k}}{v^{2}\sinh^2u}\right)u_k \nonumber\\
  = & \frac 1{\sinh^2u}u^k\left(1-\frac{|\bar{\nabla}u|^2}{v^2\sinh^2u}\right)\nonumber\\
  =& \frac {u^j}{v^2\sinh^2u}.
\end{align}
In particular, a geodesic sphere $S^n(r)$ of radius $r$ in $\mathbb{H}^{n+1}$ has principal curvatures $\lambda_i=\coth r$, $i=1,\cdots, n$.

\section{Evolution equations}\label{sec:evl}
In this section, we derive the evolution equations along the flow \eqref{eq-flow}. As in Section \ref{sec:2-3}, we can write the speed function $f(\kappa)$ as a function $F$ of the diagonalizable shfited Weingarten matrix $\mathcal{W}-I$, and equivalently we can view $F$ as the function of the shifted second fundamental form $\hat{h}_{ij}=h_{ij}-g_{ij}$ and the metric $g_{ij}$: $F=F(\hat{h}_{ij},g_{ij})=F(g^{-1/2}\hat{h}g^{-1/2},I)$ due to the $O(n)$-invariance of $F$. We write $\dot{F}^{kl}$ and $\ddot{F}^{kl,rs}$ as derivatives of $F$  with respect to the shifted second fundamental form, and $\dot{f}^i, \ddot{f}^{ij}$ the derivatives of $f$ with respect to the shifted principal curvatures. By Lemma \ref{s2:lem-conc} and Lemma \ref{inv-concave}, $f$ and $F$ have the same concavity and inverse concavity properties.

Set $\Phi(r)=-r^{-p}$ for $r>0$ and denote $\Phi'=\frac d{dr}\Phi(r)$. Then the flow equation \eqref{eq-flow} is equivalent to
\begin{equation}\label{s2:flow1}
\frac{\partial}{\partial t}X(x,t)=~-\Phi(x,t)\nu(x,t)=-\Phi(F)\nu(x,t).
\end{equation}
We first have the following evolution equations for $g_{ij}$, the unit normal vector field $\nu$ and the shifted second fundamental form $\hat{h}_{ij}$.
\begin{lem}\label{evolution}
Along the flow equation \eqref{s2:flow1}, we have
\begin{align}
\frac{\partial}{\partial t}g_{ij}=&~-2\Phi h_{ij},\label{s3:1}\\
\frac{\partial}{\partial t}\nu=&~\nabla\Phi,\label{s3:2}\\
\frac{\partial}{\partial t}\hat{h}_{ij}=& ~\nabla_j\nabla_i\Phi-\Phi(\hat{h}^2)_{ij}\label{s3:1-h}\\
\frac{\partial}{\partial t}\hat{h}_{i}^{j}=&~ \nabla^{j}\nabla_{i}\Phi+\Phi((\hat{h}^{2})_{i}^{j}+2\hat{h}_{i}^{j}),\label{s3:3}
\end{align}
where $(\hat{h}^{2})_{ij}=\hat{h}_i^k\hat{h}_{kj}$.
\end{lem}
\proof
The equations \eqref{s3:1} and \eqref{s3:2} are well-known. For the third equation, we apply the evolution equation \cite[Theorem 3-15]{And94} of $h_{ij}$ along a curvature flow in hyperbolic space and derive that
\begin{align*}
  \frac{\partial}{\partial t}\hat{h}_{ij}=&~\frac{\partial}{\partial t}(h_{ij}-g_{ij}) \\
  =&~ \nabla_j\nabla_i\Phi-\Phi((h^2)_{ij}+g_{ij})+2\Phi h_{ij}\\
  =&~\nabla_j\nabla_i\Phi-\Phi((h^2)_{ij}+g_{ij}-2h_{ij})\\
  =&~\nabla_j\nabla_i\Phi-\Phi(\hat{h}^2)_{ij}.
\end{align*}
For the last equation \eqref{s3:3}, we combine equations \eqref{s3:1} and \eqref{s3:1-h}:
\begin{align*}
  \frac{\partial}{\partial t}\hat{h}_{i}^{j}=&~\frac{\partial}{\partial t}(g^{jk}\hat{h}_{ik}) \\
  =&~2\Phi h^{jk}\hat{h}_{ik}+g^{jk}\nabla_k\nabla_i\Phi-\Phi g^{jk}(\hat{h}^2)_{ik}\\
   =&~\nabla^j\nabla_i\Phi+\Phi((\hat{h}^2)_i^j+2\hat{h}_i^j).
\end{align*}
\endproof

We define the operator
$\mathcal{L}:=\dot{\Phi}^{kl}\nabla_k\nabla_l$, then the equation \eqref{s3:3} implies a parabolic equation for the speed function $\Phi$:
\begin{lem}
The speed function $\Phi$ satisfies
\begin{equation}
(\frac{\partial }{\partial t}-\mathcal{L})\Phi=\Phi\dot{\Phi}^{kl}(\hat{h}^2)_{kl}-2p\Phi^2.\label{s2:eq-Phi}
\end{equation}
\end{lem}

We also derive a parabolic equation satisfied by the shifted second fundamental form.
\begin{lem}
The shifted second fundamental form satisfies the following evolution equation
\begin{align}\label{s3:1-ssff}
(\frac{\partial }{\partial t}-\mathcal{L})\hat{h}_{ij}=&\ddot{\Phi}^{kl,rs}\nabla_{i}\hat{h}_{kl}\nabla_{j}\hat{h}_{rs}+\dot{\Phi}^{kl}(\hat{h}^2)_{kl}\left(\hat{h}_{ij}+g_{ij}\right)\nonumber\\
&+\left((p-1)\Phi-\dot{\Phi}^{kl}g_{kl}\right)(\hat{h}^2)_{ij},
\end{align}
and the shifted Weingarten matrix satisfies
\begin{align}\label{s2:eq-sff}
(\frac{\partial }{\partial t}-\mathcal{L})\hat{h}_{i}^{j}=&\ddot{\Phi}^{kl,rs}\nabla_{i}\hat{h}_{kl}\nabla^{j}\hat{h}_{rs}+\left(\dot{\Phi}^{kl}(\hat{h}^2)_{kl}+2\Phi\right)\hat{h}_{i}^{j}+\dot{\Phi}^{kl}(\hat{h}^2)_{kl}\delta_{i}^j\nonumber\\
&+\left((p+1)\Phi-\dot{\Phi}^{kl}g_{kl}\right)(\hat{h}^2)_i^j.
\end{align}
\end{lem}
\proof
By the equation \eqref{s3:1-h},
\begin{align}\label{s3:3-pf1}
\frac{\partial}{\partial t}\hat{h}_{ij}=&~ \nabla_{j}\nabla_{i}\Phi-\Phi(\hat{h}^{2})_{ij}\nonumber\\
  = & ~\dot{\Phi}^{kl}\nabla_j\nabla_i\hat{h}_{kl}+\ddot{\Phi}^{kl,rs}\nabla_i\hat{h}_{kl}\nabla_j\hat{h}_{rs}-\Phi(\hat{h}^{2})_{ij}.
\end{align}
Combining the Gauss and Codazzi equations in hyperbolic space, we obtain the following generalized Simons' identity (see e.g.,\cite{And94}):
\begin{equation*}
  \nabla_{(i}\nabla_{j)}h_{kl}=\nabla_{(k}\nabla_{l)}h_{ij}+(h_k^ph_{pl}+g_{kl})h_{ij}-h_{kl}(h_i^ph_{pj}+g_{ij}),
\end{equation*}
where the brackets denote symmetrization. This implies that
\begin{align}\label{s3:Sim}
  \nabla_{(i}\nabla_{j)}\hat{h}_{kl}=&~\nabla_{(k}\nabla_{l)}\hat{h}_{ij}+(\hat{h}_k^r\hat{h}_{rl}+2\hat{h}_{kl}+2g_{kl})(\hat{h}_{ij}+g_{ij})\nonumber\\
  &~-(\hat{h}_{kl}+g_{kl})(\hat{h}_i^r\hat{h}_{rj}+2\hat{h}_{ij}+2g_{ij})\nonumber\\
  =&~\nabla_{(k}\nabla_{l)}\hat{h}_{ij}+\hat{h}_k^r\hat{h}_{rl}(\hat{h}_{ij}+g_{ij})-(\hat{h}_{kl}+g_{kl})\hat{h}_i^r\hat{h}_{rj}.
\end{align}
The equation \eqref{s3:1-ssff} follows by substituting \eqref{s3:Sim} into \eqref{s3:3-pf1} and using the fact $\dot{\Phi}^{kl}\hat{h}_{kl}=-p\Phi$ due to the homogeneity of $\Phi$. The second equation \eqref{s2:eq-sff} follows by combining \eqref{s3:1-ssff} and \eqref{s3:1}.
\endproof

Since the initial hypersurface $\Sigma_0$ is horo-convex in $\mathbb{H}^{n+1}$, we parametrize it as a graph  of a smooth function $u_0$ on the sphere $\mathbb{S}^n$ in the geodesic polar coordinate system centered at some point $o$ inside $\Sigma_0$. Suppose that $\Sigma_t$ is a smooth horo-convex solution to \eqref{s2:flow1}, by adding a family of time-dependent tangential diffeomorphism of $\mathbb{S}^n$, denoted by $\xi(z,t)$, we can parametrize $\Sigma_t$ as a graph on the sphere centered at $o$ in the form $X(z,t)=(\xi(z,t),u(\xi(z,t)))$, where $z\in \mathbb{S}^n$. Since $\Sigma_t$ solves the flow \eqref{s2:flow1}, we have
\begin{align}
  \frac d{dt}u(\xi(z,t)) =&~ -\Phi\nu\cdot\partial_r=-\frac {\Phi}v, \label{s3:u1}\\
  \frac d{dt}\xi^i(z,t) =&~ -\Phi\nu\cdot\partial_i=\frac{\Phi u^i}{v\sinh^2u}.
\end{align}
It follows that $u$ satisfies the following scalar parabolic equation on $\mathbb{S}^n$:
\begin{align}\label{s2:u-evl}
  \frac{\partial}{\partial t}u(\xi,t) =& \frac d{dt}u(\xi(z,t))-u_i\frac d{dt}\xi^i(z,t) \nonumber\\
  =& -\frac {\Phi}v -\frac{\Phi |Du|^2}{v\sinh^2u}\nonumber\\
  =&-v\Phi,
\end{align}
where $v$ is the function defined in \eqref{s2:v-def}, and $\Phi$ is evaluated at $(\hat{h}_i^j)=(h_i^j-\delta_i^j)$ with $h_i^j$ given by \eqref{s2:h-hat} in terms of functions of $u$.

If we regard $u$ as a function on the hypersurface $\Sigma_t$ and denote its covariant derivatives with respect to the induced metric on $\Sigma_t$ by $\nabla_iu$ and $\nabla_i\nabla_ju$, then the function $u$ satisfies the following equation.
\begin{lem}
The function $u$ satisfies the following parabolic equation on $\Sigma_t$:
\begin{align}\label{s2:eq-u}
(\frac{\partial }{\partial t}-\mathcal{L})u=&-(p+1)\Phi v^{-1}+\dot{\Phi}^{kl}g_{kl}\left(v^{-1}-\coth u\right)+\coth u\dot{\Phi}^{kl}\nabla_k u\nabla_l u.
\end{align}
\end{lem}
\proof
First we have the relation (see e.g., \cite{Ge-2011})
\begin{equation}\label{s3:h}
  h_{ij}v^{-1}=-\nabla_i\nabla_ju+\cosh u\sinh u\sigma_{ij}.
\end{equation}
In fact, since the induced metric on $\Sigma_t$ is given by $g_{ij}=\sinh^2u \sigma_{ij}+u_iu_j$, we can compute the Christoffel symbols of $g_{ij}$ as follows.
\begin{align*}
  \Gamma_{ij}^k-\bar{\Gamma}_{ij}^k= & u_{ij}g^{kl}u_l+\sinh u\cosh ug^{kl}\left(u_i\sigma_{jl}+u_j\sigma_{il}-u_l\sigma_{ij}\right),
\end{align*}
where $u_i,u_{ij}$ are derivatives of $u$ with respect to the canonical metric $\sigma_{ij}$ on $\mathbb{S}^n$ and $\bar{\Gamma}_{ij}^k$ are the Christoffel symbols of $\sigma_{ij}$. Then using \eqref{s2:uk} we have
\begin{align}\label{s3:uij}
  \nabla_i\nabla_ju= & u_{ij}+ \bar{\Gamma}_{ij}^k u_k-\Gamma_{ij}^ku_k\nonumber\\
  = &u_{ij}-u_{ij}u_ku_lg^{kl} -\sinh u\cosh u u_kg^{kl}\left(u_i\sigma_{jl}+u_j\sigma_{il}-u_l\sigma_{ij}\right)\nonumber\\
  =& u_{ij}\left(1-\frac{|\bar{\nabla}u|^2}{v^2\sinh^2u}\right)-\frac{\coth u}{v^2}\left(2u_iu_j-|\bar{\nabla}u|^2\sigma_{ij}\right)\nonumber\\
  =& \frac{u_{ij}}{v^2}-\frac{\coth u}{v^2}\left(2u_iu_j-|\bar{\nabla}u|^2\sigma_{ij}\right).
\end{align}
The equation \eqref{s3:h} follows by substituting \eqref{s3:uij} into \eqref{s2:hij}. By applying \eqref{s3:h}, we have
\begin{align*}
  \mathcal{L}u =& \dot{\Phi}^{kl}\nabla_k\nabla_lu \\
  = & -v^{-1}\dot{\Phi}^{kl}h_{kl}+\cosh u\sinh u\dot{\Phi}^{kl}\sigma_{kl}\\
  =& -v^{-1}\dot{\Phi}^{kl}(\hat{h}_{kl}+g_{kl})+\coth u\dot{\Phi}^{kl}\left(g_{kl}-u_ku_l\right)\\
  =&p\Phi v^{-1}-(v^{-1}-\coth u)\dot{\Phi}^{kl}g_{kl}-\coth u\dot{\Phi}^{kl}\nabla_ku\nabla_lu.
\end{align*}
The equation \eqref{s2:eq-u} now follows from the above equation and the total derivative \eqref{s3:u1} of $u$.
\endproof

To conclude this section, we include the tensor maximum principle, which was first proved by Hamilton \cite{Ha1982} and was generalized by Andrews \cite{BAndrews07}.
\begin{thm}[\cite{BAndrews07}]\label{s2:tensor-mp}
Let $S_{ij}$ be a smooth time-varying symmetric tensor field on a compact manifold $\Sigma$ (possibly with boundary), satisfying
\begin{equation*}
\frac{\partial}{\partial t}S_{ij}=a^{kl}\nabla_k\nabla_lS_{ij}+u^k\nabla_kS_{ij}+N_{ij},
\end{equation*}
where $a^{kl}$ and $u$ are smooth, $\nabla$ is a (possibly time-dependent) smooth symmetric connection, and $a^{kl}$ is positive definite everywhere. Suppose that
\begin{equation}\label{s2:TM2}
  N_{ij}\mu^i\mu^j+\sup_{\Gamma}2a^{kl}\left(2\Gamma_k^r\nabla_lS_{ir}\mu^i-\Gamma_k^r\Gamma_l^sS_{rs}\right)\geq 0
\end{equation}
whenever $S_{ij}\geq 0$ and $S_{ij}\mu^j=0$. If $S_{ij}$ is positive definite everywhere on $\Sigma$ at $t=0$ and on $\partial \Sigma$ for $0\leq t\leq T$, then it is positive on $\Sigma\times[0,T]$.
\end{thm}

\begin{rem}Theorem \ref{s2:tensor-mp} is the main tool to prove the pinching estimates in Section \ref{sec:pinc}.
\end{rem}

\section{First estimates}\label{sec:4}

In this section, we derive a priori $C^0$ and $C^1$ estimates for smooth horo-convex solution to the flow equation \eqref{eq-flow}. We only require that $F$ satisfies Assumption \ref{s1:assum1}. First we look at the spherical solution of the flow.  We fix a point $o\in\mathbb{H}^{n+1}$ and consider geodesic polar coordinates centered at $o$. For geodesic sphere of radius $\theta$ centered at $o$, the shifted principal curvatures are given by
\begin{equation*}
  \kappa_i=\coth\theta-1=\frac{2}{e^{2\theta}-1}.
\end{equation*}
Let $\theta(t,\theta_0)$ be the solution of
\begin{equation}\label{eq-theta}
\left\{\begin{aligned}
\frac{d}{dt}\theta=&~\frac{1}{n^p}(\frac{e^{2\theta}-1}{2})^{p},\\
\theta(0)=&~\theta_0.
\end{aligned}\right.
\end{equation}
It's easy to see that $S_t(\theta_0):=\{(\theta(t,\theta_0),\xi)\mid\xi\in\mathbb{S}^n\}$ is a solution of \eqref{eq-flow} with initial condition $S_0=\{(\theta_0,\xi)\mid\xi\in\mathbb{S}^n\}$.

Denote
\begin{equation*}
  Q(t)=\F{e^{2\theta(t)}-1}{2}.
\end{equation*}
Then $Q(t)$ is the reciprocal of the shifted principal curvature of the spherical solution with radius $\theta(t)$. By \eqref{eq-theta}, we have
\begin{equation}\label{Q}
\F{d}{dt}Q(t)=\F{1}{n^p}(2Q(t)+1)Q(t)^p.
\end{equation}
The equation \eqref{Q} implies that $Q(t)$ blows up in finite time, and therefore the spherical solution $S_t$ expands to infinity in finite time.

\subsection{$C^0$ estimate}$\ $

We write the initial hypersurface $\Sigma_0$ as a radial graph on $\mathbb{S}^n$ centered at some point $o$ in the domain enclosed by $\Sigma_0$  and  let $\bar{r}=\sup_{\Sigma_0}u$ and $\underline{r}=\inf_{\Sigma_0}u$. Since the flow \eqref{eq-flow} is parabolic, by the comparison principle, the solution $\Sigma_{t}$ of \eqref{eq-flow} is always contained in the domain bounded by the spherical barriers $S_t(\bar{r})$ and $S_t(\underline{r})$ as long as the solution exists. Since the spherical solution expands to infinity in finite time, we have that the maximal existence time $T^*$ of $\Sigma_t$ has to be finite.  Now we have the following oscillation estimate.
\begin{lem}\label{est-C0}
Let $\Sigma_t$ be a smooth, horo-convex solution of \eqref{eq-flow} on $[0,T^*)$ and write $\Sigma_t$ as graphs of the function $u(\cdot,t)$ on $\mathbb{S}^n$ centered at some point $o$. Then there exists a positive constant $C$ depending only on the initial hypersurface $\Sigma_0$ such that
\begin{equation}\label{s4:C0}
\mathrm{osc}(u(\cdot,t))~\leq ~\ln 2+C,\ \ \forall~t\in [0,T^*).
\end{equation}
\end{lem}
\begin{proof}
As long as the solution stays horo-convex, by \eqref{s2:inner} we have
\begin{equation*}
\rho_{+}(t)-\rho_{-}(t)<\ln2,
\end{equation*}
where $\rho_{+}(t)$ and $\rho_{-}(t)$ are the outer-radius and inner-radius of $\Sigma_{t}$. We denote $o^{+}(t)$ and $o^{-}(t)$ the center of the outer and inner ball, then there exist $\xi_1(t), \xi_2(t)\in\Sigma_t$ such that
\begin{align*}
\mathrm{osc}(u)=&\mathrm{d}_{\mathbb{H}}(\xi_1(t),o)-\mathrm{d}_{\mathbb{H}}(\xi_2(t),o)\\
\leq&\mathrm{d}_{\mathbb{H}}(\xi_1(t),o^{+}(t))+\mathrm{d}_{\mathbb{H}}(o^{+}(t),o)-\mathrm{d}_{\mathbb{H}}(\xi_2(t),o^{-}(t))+\mathrm{d}_{\mathbb{H}}(o^{-}(t),o)\\
\leq&\rho_{+}(t)-\rho_{-}(t)+\mathrm{d}_{\mathbb{H}}(o^{+}(t),o)+\mathrm{d}_{\mathbb{H}}(o^{-}(t),o),
\end{align*}
where $\mathrm{d}_{\mathbb{H}}(\cdot,\cdot)$ is the distance function in $\mathbb{H}^{n+1}$. Though $o^{+}(t)$ and $o^{-}(t)$ may not be unique for each $\Sigma_{t}$, we can choose them properly and prove the following claim by using the Alexandrov reflection argument:
\begin{claim}
Let $B=\overline{B_r(o)}$ be a closed outer ball of $\Sigma_0$. Then for each $\Sigma_{t}$, we can find a pair of outer and inner balls with centers $o^{+}(t)$ and $o^{-}(t)$ in $B$. Therefore we have
\begin{equation}\label{claim1}
\mathrm{d}_{\mathbb{H}}(o^{+}(t),o)+\mathrm{d}_{\mathbb{H}}(o^{-}(t),o)\leq C(\Sigma_{0}),\ \ \forall\  0<t<T^*.
\end{equation}
\end{claim}	
To prove the claim, we apply the Alexandrov-reflection argument as in \cite{A-W,Chow97}. Let $\gamma$ be a geodesic line in $\HH^{n+1}$, and let $H_{\gamma(s)}$ be the totally geodesic hyperbolic $n$-plane in $\HH^{n+1}$ which is perpendicular to $\gamma$ at $\gamma(s),\ s\in\RR$. We use the notation $H_s^+$ and $H_s^-$ for the half-spaces in $\HH^{n+1}$ determined by $H_{\gamma(s)}$:
	$$H_s^+:=\bigcup_{s'\geq s}H_{\gamma(s')},\ \ H_s^-:=\bigcup_{s'\leq s}H_{\gamma(s')}.$$
For a bounded domain $\Omega$ in $\HH^{n+1}$, denote
	$$\Omega^+(s)=\Omega\cap H_s^+,\ \ \Omega^-(s)=\Omega\cap H_s^-.$$
The reflection map across $H_{\gamma(s)}$ is denoted by $R_{\gamma,s}$. We define
\begin{equation*}
	S_\gamma^+(\Omega):=\inf\{s\in\RR\mid R_{\gamma,s}(\Omega^+(s))\subset\Omega^-(s)\}.
\end{equation*}

\begin{figure}
\begin{tikzpicture}[scale=.9]
\begin{scope}[xshift=0.5cm]
\draw[thick](-3.5,-3.4) [rounded corners=10pt] -- (-3.8,-2.6) -- (-3.6,-1.3)
 [rounded corners=12pt]--(-2.8,-0.1)
 [rounded corners=12pt]--(-1.7,-0.3)
 [rounded corners=12pt]--(-0.3,-0.7)
 [rounded corners=5pt]--(-0.1,-1.9)
 [rounded corners=12pt]--(-0.4,-2.7)--(-1.9,-3.7)
 [rounded corners=8pt]--cycle;
 \end{scope}
\begin{scope}[xshift=-5.5cm,yshift=-0.7cm,scale=1.5]
\draw[thick](1.6,-2.4) [rounded corners=20pt] -- (0.9,-0.4) -- (2.0,0.9)
 [rounded corners=15pt]--(2.7, 1.5)
 [rounded corners=11pt]--(3.4,1.4)
 [rounded corners=8pt]--(3.9,1.1)
 [rounded corners=15pt]--(5.3,0.1) -- (5.0,-1.3)
  [rounded corners=10pt]--(3.7,-2.5)--(3.1,-2.8)
  [rounded corners=18pt]--cycle;
\end{scope}

\draw[thick](0.71,-6) -- (0.71, 3) node[right]{$\Pi=H_{\gamma(s_0)}$};
\draw[->][style=thick](-5,-2) -- (6,-2) node[above]{$\gamma(s)$};
\draw[color=blue!60, thick,dashed](-1.45,-2.05) circle (2);
\draw[color=red!60, thick,dashed](-0.76,-1.77) circle (3.28);
\draw[thick] (-0.76,-1.77) circle (0.015) node[above] {$o^+(t)$};
\draw[->](-0.76,-1.77) -- (1.35,0.77);
\node at (2.0,0.4) {$\rho_+(t)$};
\node at (1.8,-1.5) {$\Omega_t$};
\node at (-1.5,-4.3) {$\partial B$};
\node at (-2.5,-2.9) {$\Omega_0$};
\end{tikzpicture}
\caption{Alexandrov reflection}
\end{figure}
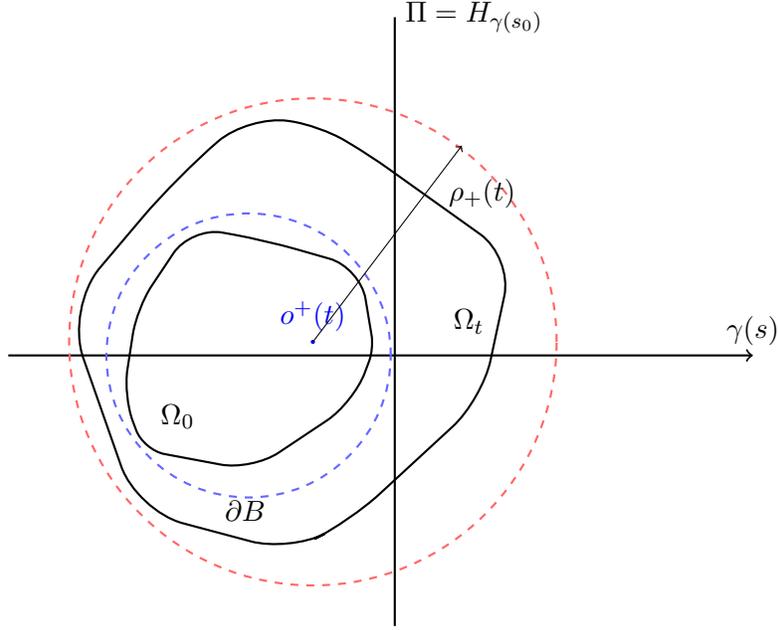

As in \cite{A-W}, for any geodesic line $\gamma$ in $\HH^{n+1}$, $S_\gamma^+(\Omega_t)$ is strictly decreasing along flow \eqref{eq-flow} unless $R_{\gamma,\bar{s}}(\Omega_t)=\Omega_t$ for some $\bar s\in\RR$. In other words, if one can reflect the domain $\Omega_0$ bounded by $\Sigma_{0}$ at a hyperplane $H_{\gamma(s_0)}$ such that $R_{\gamma,s_0}(\Omega_0^+(s_0))\subset\Omega_0^-(s_0)$, then for all $t\in[0,T^*)$, we can reflect $\Omega_{t}$ such that $R_{\gamma,s_0}(\Omega_t^+(s_0))\subset\Omega_t^-(s_0)$. Then we choose $H_{\gamma(s_0)}$ as the hyperplane which is tangential to $\partial B$, with $\gamma(s_0)\in\P B$, $\gamma'(s_0)$ point to the outer normal direction of $\P B$ at $\gamma(s_0)$. Let $S_{\rho^{+}(t)}(o^{+}(t))$ be a circumscribed sphere of $\Omega_t$ with center $o^{+}(t)$. If $o^+(t)\notin B$, i.e., $d_\mathbb{H}(o,o^+(t))=d>r$, then choose $\gamma$ as the geodesic line passing through $o$ heading towards $o^+(t)$, since $R_{\gamma,s_0}(\Omega_t^+(s_0))\subset\Omega_t^-(s_0)$, one can  check directly that $S_{\rho^{+}(t)}(R_{\gamma,s_0}(o^{+}(t)))$ is also a circumscribed sphere of $\Omega_t$, and $d_\mathbb{H}(o,R_{\gamma,s_0}(o^{+}(t)))=|d-2r|$. If $d\leq 3r$, then $R_{\gamma,s_0}(o^{+}(t))$ is the center of an outer ball we want. If $d>3r$, after repeating the reflection finite times, we can also find such an outer ball with center lying in the geodesic ball $B$. Similarly, we can also find an inner ball with center $o^-(t)$ lying in  $B$. This completes the proof of the claim and we conclude that $\mathrm{osc}(u)\leq \ln2 + C(\Sigma_{0})$.
\end{proof}

In the next two sections, we shall show that the flow will remain smooth with uniform estimates as long as it stays in a compact domain, and we will also prove that $\limsup_{t\to T^*}\max_{\mathbb{S}^n}u(\cdot,t)=+\infty$. Let $\theta_0$ be such that $T^*$ is the same maximal existence time of the spherical solution $\theta(t)=\theta(t,\theta_0)$, and we also denote $\theta(t)=\theta(t,T^*)$ if there is no confusion. For each $t\in[0,T^*)$ there exists a point $\xi_{t}\in\mathbb{S}^n$ such that $u(\xi_{t},t)=\theta(t)$.Then the estimate (4.3) implies the following Corollary.
\begin{cor}\label{newcoro4.2}
	Assuming $\limsup_{t\rightarrow T^*}\max_{\mathbb{S}^n}u(\cdot,t)=\infty$, then for all $t\in[0,T^*)$, there exists $C=C(\Sigma_0)>0$ such that
	\begin{equation*}
	|u-\theta|\leq C.
	\end{equation*}	
\end{cor}

\subsection{$C^1$ estimate}$\ $

The oscillation estimate in Lemma \ref{est-C0} together with a gradient estimate for convex graph (see \cite[Theorem 2.7.10 ]{Ge-2006}) implies that the function $v$ defined in \eqref{s2:v-def} satisfies
\begin{equation}\label{s4:v0}
  v \leq \exp {\left(\bar{\kappa}~\mathrm{osc}(u) \right)} \leq  \exp{\left(\coth\underline{r} ~\mathrm{osc}(u)\right)},
\end{equation}
where $\bar{\kappa}$ denotes the upper bound of the principal curvatures of the geodesic spheres intersecting $\Sigma_t$, and $\underline{r}=\inf_{\Sigma_t}u$. Since $\coth \underline{r}(t)\leq \coth \underline{r}(0)$ and $\mathrm{osc}(u(\cdot,t))\leq C$,  the inequality \eqref{s4:v0} gives a rough upper bound on the gradient of $u$:
\begin{equation}\label{s4:C1-0}
  \frac{|\bar{\nabla}u|^2}{\sinh^2u}\leq C
\end{equation}
for some positive constant $C$ depending only on the initial hypersurface. However, this is not sufficient for our purpose to show the asymptotical roundness of the $\Sigma_t$ as $t$ approaches the maximal existence time. Since we are considering the horo-convex solution of the flow \eqref{eq-flow}, we can use the horo-convexity of the graph $\Sigma_t=\mathrm{graph}~ u$ to refine the argument in Theorem 2.7.10 of \cite{Ge-2006} and get an improved gradient estimate on the function $u$ as follows.

\begin{lem}\label{est-C1}
Let $\Sigma_t=\mathrm{graph}~ u$ be a smooth horo-convex solution to the flow \eqref{eq-flow}. Then there exists a positive constant $C=C(\Sigma_{0})$ such that the function $v$ defined in \eqref{s2:v-def} satisfies the estimate
\begin{equation}\label{s4:C1}
v\leq \exp\left(Ce^{-2u}\right),\quad \forall~t\in [0,T^*).
\end{equation}
\end{lem}
\begin{proof}
First we recall that
\begin{equation}\label{s4:v1}
  v^2=1+\frac{|\bar{\nabla}u|^2}{\sinh^2u}=1+|\bar{\nabla}\varphi|^2,
\end{equation}
where $|\bar{\nabla}u|^2=\sigma^{ij}u_iu_j$. Inspired by the proof of Theorem 2.7.10 in \cite{Ge-2006}, we define a function $G=\ln(v)+ku$, where $k$ is a constant to be determined. At  the maximum point $\xi_{0}$ of $G$, we have
\begin{equation}\label{s4:G1}
  0=G_i=v^{-1}v_i+ku_i.
\end{equation}
Differentiating \eqref{s4:v1} and applying \eqref{s4:G1}, we have
\begin{equation}\label{s4:G2}
  0=\varphi^k\varphi_{ki}+kv^2u_i,
\end{equation}
where $\varphi_{ki}$ is the second covariant derivative of $\varphi$ with respect to the metric $\sigma_{ij}$ on $\mathbb{S}^n$, and $\varphi^k=\sigma^{kl}\varphi_l$. Multiplying \eqref{s4:G2} by $u^i$, taking the sum of $i$ from $1$ to $n$ and by using the equation \eqref{s2:h-hat0}, we obtain that
\begin{align}\label{s4:G3}
  0= & \frac{v}{\sinh u}\varphi^k\left(\frac{\coth u}v\delta_i^j-h_i^j\right)g_{kj}u^i+kv^{2}|\bar{\nabla}u|^2.
\end{align}
We observe that
\begin{align*}
 \varphi^kg_{kj}=&\frac{u^k}{\sinh u}\sinh^2u\left(\sigma_{kj}+\frac{u_ku_j}{\sinh^2u}\right)\\
 =&\sinh u \left(u_j+\frac{|\bar{\nabla}u|^2u_j}{\sinh^2u}\right)\\
 =&v^2\sinh u u_j.
\end{align*}
Since $\Sigma_t$ is horo-convex, we have $h_{i}^j>\delta_{i}^j$. Then \eqref{s4:G3} implies that
\begin{equation}\label{s4:v2}
0\leq \left(\coth u-v+k\right)v^2|\bar{\nabla}u|^2
\end{equation}
holds at the maximum point $\xi_{0}$ of $G$.  Set $\underline{r}=\inf_{\Sigma_t}u$. By choosing $k=-\coth \underline{r}+1$, the inequality  \eqref{s4:v2} implies that
\begin{align*}
  0 \leq  & \left(\coth \underline{r}-v+k\right)v^2|\bar{\nabla}u|^2  =  -(v-1)v^2|\bar{\nabla}u|^2
\end{align*}
holds at the maximum point $\xi_{0}$ of $G$. Then we have $|\bar{\nabla}u|^2(\xi_0)=0$.

Hence
\begin{align*}
G\leq \max_{\mathbb{S}^n}G=&G(\xi_{0})=ku(\xi_{0})
\end{align*}
and
\begin{align}\label{s4:2-v}
  v(\xi)\leq & \exp (k(u(\xi_0)-u(\xi))\nonumber\\
  \leq & \exp ((\coth \underline{r}-1)\mathrm{osc}(u))\nonumber\\
\leq&\exp\left(\frac{C}{e^{2\underline{r}}-1}\right)\nonumber\\
\leq&\exp\left(Ce^{-2u}\right)
\end{align}
for any $\xi$ on $\mathbb{S}^n$, where we used $u-\underline{r}\leq \mathrm{osc}(u)\leq C$ from the proof of Lemma \ref{est-C0}.
\end{proof}

Therefore, the estimate \eqref{s4:C1} implies that
\begin{equation*}
  |\bar{\nabla}u|~\leq ~C,\qquad \forall ~t\in [0,T^*),
\end{equation*}
which is much better than \eqref{s4:C1-0}, as we shall show in the next two sections that under the assumptions of Theorem  \ref{thm-0<p<infty} and Theorem \ref{thm-0<p<1}, we have
\begin{equation*}
  \limsup_{t\to T^*}\max_{\mathbb{S}^n}u(\cdot,t)=\infty.
\end{equation*}

\section{Proof of Theorem \ref{thm-0<p<infty}}\label{sec:thm1}

In this section, we prove Theorem \ref{thm-0<p<infty}. We first prove that the maximal existence time $T^*$ is characterized by the blow up time of the flow.

\subsection{Maximal existence}
\begin{prop}\label{s5:prop1}
Under the assumption of Theorem \ref{thm-0<p<infty}, the solution $\Sigma_t=\mathrm{graph} ~u(\cdot,t)$ of the flow \eqref{eq-flow} satisfies
\begin{equation*}
  \limsup_{t\to T^*}\max_{\mathbb{S}^n}u(\cdot,t)=\infty.
\end{equation*}
\end{prop}

To prove Proposition \ref{s5:prop1}, we first need the upper bound of the shifted principal curvatures of $\Sigma_t$.
\begin{lem}\label{s5:lem1}
Let $F$ be a concave function satisfying Assumption \ref{s1:assum1} and $\Sigma_t$ be a smooth horo-convex solution to the flow \eqref{eq-flow} for $t\in [0,T)$ with $T\leq T^*$. Then there is  positive constant $C$ determined by $p$ and $\Sigma_{0}$ such that the shifted principal curvatures satisfy that
\begin{equation*}
 \kappa_i\leq C
\end{equation*}
along the flow \eqref{eq-flow} for all $t\leq T$.
\end{lem}
\begin{proof}
Define a function
\begin{equation}\label{s5:1-zeta}
  \zeta(x,t)=\sup\{\hat{h}_{ij}\eta^i\eta^j:~g_{ij}\eta^i\eta^j=1\},
\end{equation}
which is the largest shifted principal curvature of $\Sigma_t$ at $(x,t)$. For any time $t_0\in [0,T)$, we assume that the maximum of $\zeta(\cdot,t)$ is achieved at the point $x_0$ in the direction $\eta=e_n$, where we choose an orthonormal frame $\{e_1,\cdots,e_n\}$ at $(x_0,t_0)$. Since $F$ is concave, we have $\Phi=-F^{-p}$ is concave. The equation  \eqref{s2:eq-sff} implies that
\begin{align}\label{s5:1}
	(\frac{\partial }{\partial t}-\mathcal{L})\zeta\leq &\left(pF^{-p-1}\sum_k\dot{f}^{k}\kappa^2_{k}-2F^{-p}\right)\zeta+pF^{-p-1}\sum_k\dot{f}^{k}\kappa^2_{k}\nonumber\\
	&+\left(-(p+1)F^{-p} -pF^{-p-1}\sum_k\dot{f}^k\right)\zeta^2.
\end{align}
By the homogeneity of $F$ and the definition of $\zeta$, we have
\begin{equation*}
  \dot{f}^{k}\kappa^2_{k}\leq \zeta \dot{f}^{k}\kappa_{k}=\zeta F
\end{equation*}
and
\begin{equation*}
  \dot{f}^{k}\kappa^2_{k}\leq \zeta^2\sum_{k=1}^n\dot{f}^{k}.
\end{equation*}
Then the inequality \eqref{s5:1} implies that
\begin{align}\label{s5:2}
	(\frac{\partial }{\partial t}-\mathcal{L})\zeta\leq &\left(pF^{-p}\zeta-2F^{-p}\right)\zeta+pF^{-p-1}\zeta^2\sum_k\dot{f}^{k}\nonumber\\
	&+\left(-(p+1)F^{-p} -pF^{-p-1}\sum_k\dot{f}^k\right)\zeta^2\nonumber\\
	=&-F^{-p}\zeta^2-2F^{-p}\zeta.
\end{align}
Applying the maximum principle to \eqref{s5:2}, we obtain that $\zeta$ is uniformly bounded from above.
\end{proof}

\begin{lem}\label{s5:lem2}
Let $F$ be a concave function satisfying Assumption \ref{s1:assum1} and $\Sigma_t$ be a smooth horo-convex solution to the flow \eqref{eq-flow}. Let $T\in [0,T^*)$ be arbitrary and assume that
\begin{equation*}
  u\leq \bar{r},\qquad \forall~0\leq t\leq T.
\end{equation*}
Then there exists a constant $C=C(p,\Sigma_0,\bar{r})$ independent of $T$ such that
\begin{equation*}
  F\geq C>0,\qquad \forall~0\leq t\leq T.
\end{equation*}
\end{lem}
\proof
We write $\Sigma_t$ as the graph of the function $u$ and define
$$G=\ln(-\Phi)+ke^{u},$$
where $k$ is a positive constant to be determined later. We first calculate the evolution equation of $G$:
\begin{align*}
  \frac{\partial }{\partial t}G =& \Phi^{-1}\frac{\partial }{\partial t}\Phi+ke^{u}\frac{\partial }{\partial t}u\\
\nabla_{i}G=&\Phi^{-1}\nabla_{i}\Phi+ke^{u}\nabla_{i}u\\
\nabla_{i}\nabla_{j}G=&\Phi^{-1}\nabla_{i}\nabla_{j}\Phi-\Phi^{-2}\nabla_{i}\Phi\nabla_{j}\Phi+ke^{u}(\nabla_{i}\nabla_{j}u+\nabla_{i}u\nabla_{j}u)\\
	=&\Phi^{-1}\nabla_{i}\nabla_{j}\Phi-(\nabla_iG-ke^{u}\nabla_iu)(\nabla_jG-ke^{u}\nabla_ju)\\
&\quad +ke^{u}(\nabla_{i}\nabla_{j}u+\nabla_{i}u\nabla_{j}u).
\end{align*}
Combining the above equations with \eqref{s2:eq-Phi} and \eqref{s2:eq-u}, we have
\begin{align}
	\left(\frac{\partial }{\partial t}-\mathcal{L}\right)G=& \Phi^{-1}\left(\frac{\partial }{\partial t}-\mathcal{L}\right)\Phi+ke^{u}\left(\frac{\partial }{\partial t}-\mathcal{L}\right)u\nonumber\\
&\quad +\dot{\Phi}^{ij}(\nabla_iG-ke^{u}\nabla_iu)(\nabla_jG-ke^{u}\nabla_ju)\nonumber\\
&\quad -ke^{u}\dot{\Phi}^{ij}\nabla_{i}u\nabla_{j}u\nonumber\\
=& \dot{\Phi}^{kl}(\hat{h}^2)_{kl}-2p\Phi +ke^{u}\biggl(-(p+1)\Phi v^{-1}+\dot{\Phi}^{kl}g_{kl}\left(v^{-1}-\coth u\right)\nonumber\\
&\quad +\coth u\dot{\Phi}^{kl}\nabla_k u\nabla_l u\biggr)-ke^{u}\dot{\Phi}^{ij}\nabla_{i}u\nabla_{j}u\nonumber\\
&\quad +\dot{\Phi}^{ij}(\nabla_iG-ke^{u}\nabla_iu)(\nabla_jG-ke^{u}\nabla_ju).
\end{align}
Looking at the increasing supremum point $(x_0,t_0)$ of $G$, i.e., $G(x_0,t_0)=\sup\limits_{\Sigma\times(0,t_0]}G$, we have $\nabla G=0$ and $(\partial_t-\mathcal{L})G\geq 0$ at $(x_0,t_0)$.
Then
\begin{align}\label{s5:1-1}
0\leq~\left(\frac{\partial }{\partial t}-\mathcal{L}\right)G\leq & pF^{-p-1}\dot{f}^{k}\kappa_k^2+2pF^{-p} +(p+1)ke^uF^{-p}v^{-1}\nonumber\\
&\quad +kpe^{u}F^{-p-1}\sum_k\dot{f}^{k}\left(v^{-1}-\coth u\right)\nonumber\\
&\quad +kpe^{u}F^{-p-1}\dot{F}^{ij}\nabla_iu\nabla_ju\left(\coth u-1+ke^u\right)
\end{align}
holds at $(x_0,t_0)$. We will take $k\leq \frac{1}{2}$, then the first line of \eqref{s5:1-1} is bounded from above by $C(p,\bar{r},\Sigma_0)F^{-p}$, since $\dot{f}^{k}\kappa_k^2\leq CF$ by Lemma \ref{s5:lem1}, $v\geq 1$ and we assumed that $u\leq \bar{r}$. By the definition \eqref{s2:v-def} and the relation \eqref{s2:uk}, we can estimate that
 \begin{align}\label{s5:4}
   \dot{F}^{ij}\nabla_iu\nabla_ju\leq & (\sum_k\dot{f}^{k})g^{ij}u_iu_j = \sum_k\dot{f}^{k} \frac {|\bar{\nabla}u|^2}{v^2\sinh^2u}\nonumber\\
   =& \sum_k\dot{f}^{k} \frac{v^2-1}{v^2}.
 \end{align}
 Substituting \eqref{s5:4} into \eqref{s5:1-1}, we have
 \begin{align}\label{s5:1-2}
0\leq &~CF^{-p} +kpe^{u}F^{-p-1}\sum_k\dot{f}^{k}\biggl(v^{-1}-\coth u+\frac{v^2-1}{v^2}\left(\coth u-1+ke^u\right)\biggr)\nonumber\\
=&CF^{-p}-kpe^uF^{-p-1}v^{-2}(\coth u-1)\sum_k\dot{f}^{k}\nonumber\\
&\quad +kpe^{u}F^{-p-1}(1-v^{-2})\left(ke^u-\frac v{v+1}\right)\sum_k\dot{f}^{k}
\end{align}
holds at $(x_0,t_0)$. Since $u\leq \bar{r}$ and $v\geq 1$, we can choose $k=\frac 12e^{-\bar{r}}$ such that
\begin{equation*}
  ke^u-\frac v{v+1}= \frac 12e^{u-\bar{r}}-\frac v{v+1}\leq \frac{1-v}{2(v+1)}\leq 0
\end{equation*}
and therefore the last line of \eqref{s5:1-2} is non-positive. Since $F$ is concave, we have
\begin{equation}\label{s5:1-4}
  \sum_k\dot{f}^{k}\geq n.
\end{equation}
Then we have that
 \begin{align}\label{s5:1-3}
0\leq &~CF^{-p}-kpne^uF^{-p-1}v^{-2}(\coth u-1)\nonumber\\
=&~CF^{-p}-knpF^{-p-1}\frac 1{v^2\sinh u}\nonumber\\
\leq &~CF^{-p}-CnpF^{-p-1}\frac 1{e^{\bar{r}}\sinh \bar{r}}
\end{align}
at $(x_0,t_0)$,  which is equivalent to $F\geq C(p,\bar{r},\Sigma_0)>0$ at $(x_0,t_0)$. In the last inequality of \eqref{s5:1-3}, we used the upper bound on $v$ in Lemma \ref{est-C1}. This implies that
\begin{equation*}
  F(x,t)\geq F(x_0,t_0)\exp\left(\frac kp(e^{u(x,t)}-e^{u(x_0,t_0)})\right)\geq C(p,\bar{r},\Sigma_0)>0
\end{equation*}
for all $(x,t)\in \Sigma_t\times [0,T]$.
\endproof
\begin{rem}\label{s5:rem1}
In the proof of Lemma \ref{s5:lem2}, the concavity of $F$ is only used to show \eqref{s5:1-4}. Lemma \ref{s5:lem1} is used to prove that $\dot{f}^{k}\kappa_k^2\leq CF$. If $F$ and $\kappa_i$ satisfy that
\begin{equation}\label{s5:1-5}
  \sum_k\dot{f}^{k}\geq C>0,~\dot{f}^{k}\kappa_k^2\leq CF
\end{equation}
for some positive constant $C$, then the conclusion of Lemma \ref{s5:lem2} is still true.
\end{rem}

Now we can complete the proof of Proposition \ref{s5:prop1}.
\begin{proof}[Proof of Proposition \ref{s5:prop1}]
For any $T<T^*$,  suppose that $u\leq \bar{r}$ for all $t\in [0,T]$. Let $0<T_h\leq T$ be the maximal time such that horo-convexity is preserved. We have proved that $F\geq C(p,\bar{r},\Sigma_0)>0$ for all $t\in [0,T_h]$. By the upper bound on the shifted principal curvatures in Lemma \ref{s5:lem1} and the assumption $F|_{\partial \Gamma_+}=0$, we derive that $\kappa_i\geq C(p,\bar{r},\Sigma_0)>0$ for all $t\in [0,T_h]$. This gives us uniform $C^2$ estimates on $u$.  Then the scalar equation \eqref{s2:u-evl} is a nonlinear uniformly parabolic equation with concave elliptic part. We can apply Krylov-Safonov's theorem \cite{Kry} to yield the uniform H\"{o}lder estimates on $u_t$ and $\bar{\nabla}^2u$. The parabolic Schauder estimate \cite{lieb1996} can be applied to derive all higher order estimates on $u$. Then we can continue to expand the flow smoothly with positive shifted principal curvature. Hence the flow preserves horo-convexity along $[0,T]$. Then the same argument shows that $u$ can be expanded smoothly as long as it is bounded. This is equivalent to that $\limsup_{t\to T^*}\max u=\infty$.
\end{proof}

\subsection{Asymptotical behavior}$\ $

Let $\theta(t)=\theta(t,T^*)$ be the radius of the spherical solution to the flow \eqref{eq-flow} with the same maximal existence time $T^*$ with $\Sigma_t$. Recall the definition (1.5) of $Q(t)$ and that $\limsup_{t\rightarrow T^*}\max u=\infty$ holds  by Proposition \ref{s5:prop1}. We prove the following estimate on the rescaled quantity $FQ$ under the assumption of Theorem \ref{thm-0<p<infty}.

\begin{lem}\label{est-FQ-lower}
Let $F$ be a concave function satisfying Assumption 1.1. Then there exists a positive constant $C$ depending only on $p$ and $\Sigma_0$ such that
\begin{equation*}
  FQ~\geq ~C,\qquad \forall~t\in [0,T^*)
\end{equation*}
along the flow \eqref{eq-flow}.
\end{lem}
\begin{proof}
We write $\Sigma_t$ as the graph of the function $u$ and define
$$G=\ln(-\Phi)+ke^{u-\theta}-p\ln(Q),$$
where $0<k<1$ is a constant to be determined later. By Corollary \ref{newcoro4.2}  we have that $ke^{u-\theta}$ is bounded. Combining the equations \eqref{s2:eq-Phi}, \eqref{s2:eq-u} and \eqref{Q}, we have
\begin{align*}
	\left(\frac{\partial }{\partial t}-\mathcal{L}\right)G=& \Phi^{-1}\left(\frac{\partial }{\partial t}-\mathcal{L}\right)\Phi+ke^{u-\theta}\left(\frac{\partial }{\partial t}-\mathcal{L}\right)u-ke^{u-\theta}\F{d}{dt}\theta-pQ^{-1}\F{d}{dt}Q\nonumber\\
&\quad +\dot{\Phi}^{ij}(\nabla_iG-ke^{u-\theta}\nabla_iu)(\nabla_jG-ke^{u-\theta}\nabla_ju)-ke^{u-\theta}\dot{\Phi}^{ij}\nabla_{i}u\nabla_{j}u\nonumber\\
=& \dot{\Phi}^{kl}(\hat{h}^2)_{kl}-2p\Phi +ke^{u-\theta}\biggl(-(p+1)\Phi v^{-1}+\dot{\Phi}^{kl}g_{kl}\left(v^{-1}-\coth u\right)\nonumber\\
&\quad +\coth u\dot{\Phi}^{kl}\nabla_k u\nabla_l u\biggr)-ke^{u-\theta}n^{-p}Q^p-pQ^{p-1}n^{-p}(2Q+1)\nonumber\\
&\quad +\dot{\Phi}^{ij}(\nabla_iG-ke^{u-\theta}\nabla_iu)(\nabla_jG-ke^{u-\theta}\nabla_ju)-ke^{u-\theta}\dot{\Phi}^{ij}\nabla_{i}u\nabla_{j}u.
\end{align*}
For any $(x_0,t_0)$ where a new space-time maximum of $G$ is achieved at some point $(x_0,t_0)$, i.e., $G(x_0,t_0)=\sup\limits_{\Sigma\times(0,t_0]}G$, we have
\begin{align}\label{s5:5}
	0\leq&\left(\frac{\partial }{\partial t}-\mathcal{L}\right)G \nonumber\\
=&\dot{\Phi}^{kl}(\hat{h}^2)_{kl}-2p\Phi +ke^{u-\theta}\left((p+1)F^{-p}v^{-1}+\dot{\Phi}^{kl}g_{kl}\left(v^{-1}-\coth u\right)\right)\nonumber\\
&\quad -ke^{u-\theta}n^{-p}Q^p-pQ^{p-1}n^{-p}(2Q+1)+pk^2F^{-p-1}e^{2u-2\theta}\dot{F}^{ij}\nabla_iu\nabla_ju \nonumber\\
&\quad +pkF^{-p-1}e^{u-\theta}(\coth u-1)\dot{F}^{ij}\nabla_{i}u\nabla_{j}u\nonumber\\
\leq &pCF^{-p} +(p+1) ke^{u-\theta}F^{-p}v^{-1} -ke^{u-\theta}n^{-p}Q^p-pQ^{p-1}n^{-p}(2Q+1) \nonumber\\
&\quad +pF^{-p-1}ke^{u-\theta}\left(v^{-1}-\coth u+(ke^{u-\theta}+\coth u-1)(1-v^{-2})\right)\sum_k\dot{f}^{k}\nonumber\\
\leq &CF^{-p} -pF^{-p-1}ke^{u-\theta}(\coth u-1)\sum_k\dot{f}^{k}\nonumber\\
&\quad +pF^{-p-1}ke^{u-\theta}\left(v^{-1}-1+(ke^{u-\theta}+\coth u-1)(1-v^{-2})\right)\sum_k\dot{f}^{k},
\end{align}
where we used \eqref{s5:4} and that $(p+1)ke^{u-\theta}F^{-p}v^{-1}$ can be absorbed in $CF^{-p}$ by applying the $C^0$ and $C^1$ estimates when $k\leq1$. Since we only care about the asymptotical behavior of the flow, we may focus on the flow with $t_0$ sufficiently close to $T^*$. We observe that
\begin{align*}
  \coth u-1 =& \frac 2{e^{2u}-1} = O(e^{-2\theta})  =O(Q^{-1}),
\end{align*}
where we used the fact that $|u-\theta|$ is bounded and the conclusion in Proposition \ref{s5:prop1}.  By the concavity of $F$, we have
\begin{equation*}
  \sum_{k=1}^n\dot{f}^{k}\geq n.
\end{equation*}
Then the second term on the right hand side of \eqref{s5:5} can be estimated as
\begin{equation*}
-pF^{-p-1}ke^{u-\theta}(\coth u-1)\sum_k\dot{f}^{k}\leq -CkF^{-p-1}Q^{-1}.
\end{equation*}
We can choose $0<k<1$ sufficiently small and depending only on $\Sigma_0$ such that the bracket in last term of \eqref{s5:5} is negative:
\begin{align*}
  & \left(v^{-1}-1+(ke^{u-\theta}+\coth u-1)(1-v^{-2})\right)\\
= & \frac{v^2-1}{v^2}\left(ke^{u-\theta}+\coth u-1-\frac v{v+1}\right)\leq ~0
\end{align*}
since $v\geq 1$, $\coth u-1$ approaches zero as $t_0$ sufficiently close to $T^*$, and $e^{u-\theta}$ is bounded by Corollary \ref{newcoro4.2}. Hence we conclude that
	$$0\leq\frac{1}{F^{p}}-\frac{C}{F^{p+1}Q}$$
holds at the maximum point $(x_0,t_0)$ of $G$ for some $C>0$ depending only on $\Sigma_0$ and $p$. Equivalently, $FQ\geq C$ at $(x_0,t_0)$.  This together with $ G(x,t)\leq G(x_0,t_0)$ for any $x\in \Sigma_t$ with $t\leq t_0$ implies that
\begin{equation*}
  (FQ)|_{(x,t)}\geq (FQ)|_{(x_0,t_0)}\exp\left({\frac kp\left(e^{u(x,t)-\theta(t)}-e^{u(x_0,t_0)-\theta(t_0)}\right)}\right)\geq C.
\end{equation*}
Since $t_0$ is arbitrary, we conclude that $FQ\geq C>0$ for some positive constant $C$ depending only on $p$ and $\Sigma_0$ along the flow \eqref{eq-flow}.
\end{proof}

\begin{lem}\label{est-FQ-upper}
Let $F$ be a function satisfying Assumption \ref{s1:assum1}. Then there exists a positive constant $C$ depending only on $p, \Sigma_0$ such that $FQ\leq C$ holds along the flow \eqref{eq-flow}.
\end{lem}
\begin{proof}
Define
$$G=-\ln(-\Phi)+ku-(\frac{k}{2}-p)\ln Q,$$
where $k>\max\{2p,\sup_{\Sigma_0}\coth u\}$ is a positive constant. Suppose that $t_0<T^*$ is a time such that a new space-time maximum of $G$ is achieved at some point $x_0\in \Sigma_{t_0}$, i.e., $G(x_0,t_0)=\sup_{\Sigma\times(0,t_0]}G$. Then we have
\begin{align*}
0\leq\left(\frac{\partial }{\partial t}-\mathcal{L}\right) G=&-\dot{\Phi}^{kl}(\hat{h}^2)_{kl}+2p\Phi+\F{k(1+p)}{F^pv}-k^2\dot{\Phi}^{kl}\nabla_k u\nabla_l u\nonumber\\
&+k\dot{\Phi}^{kl}g_{kl}(v^{-1}-\coth u)+k\coth u\dot{\Phi}^{kl}\nabla_ku\nabla_lu\nonumber \\
&-(\F{k}{2}-p)\F{1}{n^p}(2Q+1)Q^{p-1}\nonumber \\
\leq&\F{C}{F^p}-k(k-\coth u)\dot{\Phi}^{kl}\nabla_ku\nabla_lu-(\frac{k}{2}-p)\frac{2}{n^p}Q^p\\\nonumber
\leq&\F{C}{F^p}-(\frac{k}{2}-p)\frac{2}{n^p}Q^p.
\end{align*}
Hence at the maximum point $(x_0,t_0)$ of $G$, $FQ\leq C$ for some positive constant $C$. Using the definition of $G$, we conclude that $FQ\leq C$ holds along the flow.
\end{proof}

\begin{lem}\label{est-kiQ-upper}
Let $F$ be a concave function satisfying Assumption \ref{s1:assum1}.  Along the flow \eqref{eq-flow}, we have $\kappa_i Q\leq C$ for some positive constant $C$ depending on $p, \Sigma_{0}$.
\end{lem}
\begin{proof}
Recall the definition \eqref{s5:1-zeta} of the function $\zeta$. Here we consider the rescaled function $\tilde{\zeta}=\zeta Q$. Combining \eqref{s5:1} and \eqref{Q}, we have
\begin{align*}
	(\frac{\partial }{\partial t}-\mathcal{L})\tilde{\zeta}\leq &\left(pF^{-p-1}\sum_k\dot{f}^{k}\kappa^2_{k}-2F^{-p}\right)\zeta Q+pF^{-p-1}Q\sum_k\dot{f}^{k}\kappa^2_{k}\nonumber\\
	&+\left(-(p+1)F^{-p} -pF^{-p-1}\sum_k\dot{f}^k\right)\zeta^2Q+n^{-p}(2Q+1)Q^p\zeta.
\end{align*}
Since $\sum_k\dot{f}^k\geq n$, $FQ\leq C$ and $\sum_k\dot{f}^{k}\kappa^2_{k}\leq F\zeta$, we have
 \begin{align}\label{s5:2zeta-2}
	(\frac{\partial }{\partial t}-\mathcal{L})\tilde{\zeta}\leq &\left(-npC\tilde{\zeta}^2+C\tilde{\zeta}\right)Q^p.
\end{align}
Applying maximum principle to \eqref{s5:2zeta-2}, we conclude that $\tilde{\zeta}$ is uniformly bounded from above.
\end{proof}

Now we have $\kappa_iQ\leq C$ and $FQ\geq C$ along the flow \eqref{eq-flow}. By assumption $F|_{\P\Gamma^{+}}=0$, we obtain the lower bound $\kappa_iQ\geq C>0$ for some constant $C$. This is equivalent to $C^2$ bound on $u-\theta$. As in Section \ref{sec:1}, we introduce a new time parameter $\tau=\tau(t)$ by the relation
\begin{equation}\label{s5:tau}
\F{d\tau}{dt}=Q(t)^p
\end{equation}
such that $\tau(0)=0$. Then by \eqref{eq-theta},
\begin{equation}\label{s5:dtau-theta}
  \frac d{d\tau}\theta=\frac d{dt}\theta\cdot \frac{dt}{d\tau}=\frac 1{n^{p}},
\end{equation}
which means that
\begin{equation}\label{s5:thet-tau}
  \theta(\tau)-\theta_0=n^{-p}\tau
\end{equation}
and $\tau$ ranges from $0$ to $\infty$. Consider $u=u(\cdot,\tau)$ as a function on $\mathbb{S}^n\times[0,\infty)$, then
\begin{equation}\label{eq-u-theta}
\F{\P }{\P\tau}(u-\theta)=\frac{v}{F(\hat{h}_i^jQ)^p}-\frac{1}{n^p}.
\end{equation}
The uniform bounds on $\kappa_iQ$ implies that $F(\hat{h}_i^jQ)$ is a uniformly elliptic differential operator for the function $u-\theta$. By Corollary \ref{newcoro4.2} and Lemma \ref{est-C1}, there exists a uniform constant $C>0$ such that
\begin{equation*}
|u-\theta|\leq C,\quad |\bar{\nabla}(u-\theta)|^2=\sinh^2u|\bar{\nabla}\varphi|^2\leq C.
\end{equation*}
From \eqref{s2:h-hat} and the relation that $\bar{\nabla}_i\varphi=\bar{\nabla}_iu/\sinh u$, one can compute that
\begin{equation}\label{u_ij}
\bar{\nabla}_i\bar{\nabla}_kug^{kj}Q=-v\hat{h}_i^jQ+v^{-2}\coth u\bar{\nabla}_i\varphi \bar{\nabla}^j\varphi Q+\delta_i^j(\coth u-v)Q.
\end{equation}
It's not difficult to see that by using \eqref{g_ij}, \eqref{u_ij}, the bound on $u-\theta,\ v=\sqrt{1+|\bar{\nabla}\varphi|^2}$ and $\kappa_iQ$, we have
\begin{equation*}
|\bar{\nabla}^2(u-\theta)|\leq C
\end{equation*}
for some uniform constant $C>0$. Since $F$ is a concave function and $p>0$, the right hand side of \eqref{eq-u-theta} is concave with respect to the second spatial derivatives of $u-\theta$. We can apply the Krylov-Safonov estimate \cite{Kry} to derive the H\"{o}lder estimates on $(u-\theta)_{\tau}$ and $\bar{\nabla}^2(u-\theta)$, and the parabolic Schauder estimate \cite{lieb1996} to derive the $C^{k,\alpha}$ estimates of $u-\theta$ for all $k\geq 2$. This completes the proof of Theorem \ref{thm-0<p<infty}.

\section{Proof of Theorem \ref{thm-0<p<1}: Pinching estimates}\label{sec:pinc}

In this section, we prove the maximal existence of the solution to flow \eqref{eq-flow} under the assumption of Theorem \ref{thm-0<p<1}, and show that the ratio of the largest shifted principal curvature $\kappa_n$ to the smallest one $\kappa_1$ converges to $1$ as $t\to T^*$, where $T^*$ is the maximal existence time of $\Sigma_t$.
\subsection{Pinching estimates}$\ $
\begin{lem}\label{pinch-p<1}
	Under the assumptions in Theorem \ref{thm-0<p<1},  there exists a positive constant $C$, such that along the flow \eqref{eq-flow} we have
	\begin{equation}\label{est-pinch-p<1}
	\kappa_n\leq C\kappa_1\ \qquad \forall ~t\in [0,T^*),
	\end{equation}
where $\kappa_1\leq \cdots\leq \kappa_n$ are the shifted principal curvatures of $\Sigma_t$ in the increasing order.
\end{lem}
\begin{proof} We consider the four cases of the function $F$ separately. The main tool is the tensor maximum principle in Theorem \ref{s2:tensor-mp}.

(a). $F$ is concave and $F|_{\partial\Gamma_+}=0$. The conclusion \eqref{est-pinch-p<1} in this case has been included in Theorem \ref{thm-0<p<infty}. We give a direct proof of \eqref{est-pinch-p<1} here, and the calculation will be used in Lemma \ref{est-umbilic} again to obtain the improved pinching estimate. Consider the tensor $S_{ij}:=\epsilon Fg_{ij}-\hat{h}_{ij}$, where $\epsilon\geq 1/n$ is chosen such that $S_{ij}$ is positive definite initially. By the evolution equation \eqref{s2:eq-Phi}, we derive that
\begin{align}\label{s6:F}
  \left(\frac{\partial }{\partial t} -\mathcal{L}\right)F =& -p(p+1)F^{-p-2}\dot{F}^{kl}\nabla_kF\nabla_lF-F^{-p}\dot{F}^{kl}(\hat{h}^2)_{kl}-2F^{-p+1}.
\end{align}
Combining \eqref{s3:1}, \eqref{s3:1-ssff} and \eqref{s6:F}, we have
\begin{align}\label{s6:1-a1}
\left(\frac{\partial }{\partial t} -\mathcal{L}\right)S_{ij}=&\epsilon \left(\frac{\partial }{\partial t} -\mathcal{L}\right)Fg_{ij}+\epsilon F\frac{\partial }{\partial t}g_{ij}-\left(\frac{\partial }{\partial t} -\mathcal{L}\right)\hat{h}_{ij}\nonumber\\
=&-p(p+1)\epsilon F^{-p-2}\dot{F}^{kl}\nabla_kF\nabla_lFg_{ij}-\epsilon F^{-p}\dot{F}^{kl}(\hat{h}^2)_{kl}g_{ij}\nonumber\\
&-2\epsilon F^{-p+1}g_{ij}-2\epsilon F\Phi(-S_{ij}+\epsilon Fg_{ij}+g_{ij})\nonumber\\
&- \ddot{\Phi}^{kl,rs}\nabla_{i}\hat{h}_{kl}\nabla_{j}\hat{h}_{rs}-\dot{\Phi}^{kl}(\hat{h}^2)_{kl}\left(-S_{ij}+\epsilon Fg_{ij}+g_{ij}\right)\nonumber\\
&-\left((p-1)\Phi-\dot{\Phi}^{kl}g_{kl}\right)\left((S^2)_{ij}+\epsilon^2F^2g_{ij}-2\epsilon FS_{ij}\right)\nonumber\\
=&pF^{-p-1}\biggl((p+1)F^{-1}\left(\nabla_iF\nabla_jF-\epsilon g_{ij}\dot{F}^{kl}\nabla_kF\nabla_lF\right)-\ddot{F}^{kl,rs}\nabla_i\hat{h}_{kl}\nabla_j\hat{h}_{rs}\biggr) \nonumber \\
&+\left(\dot{\Phi}^{kl}g_{kl}-(p-1)\Phi\right)(S^2)_{ij}+\left(\dot{\Phi}^{kl}(\hat{h}^2)_{kl}+2\epsilon F(p\Phi-\dot{\Phi}^{kl}g_{kl})\right)S_{ij}\nonumber\\
&+(p+1)\epsilon F^{-p}\left(\epsilon F^2-\dot{F}^{kl}(\hat{h}^2)_{kl}\right)g_{ij}\nonumber\\
&+pF^{-p-1}\left(\epsilon^2F^{2}\dot{F}^{kl}g_{kl}-\dot{F}^{kl}(\hat{h}^2)_{kl}\right)g_{ij}.
\end{align}
We apply the tensor maximum principle to show that if $S_{ij}\geq0$ initially, then it remains true for all $t\in[0,T^*)$. Suppose that there exists a time $t_0\in [0,T^*)$ such that $S_{ij}\geq0$ for all $t\in [0,t_0]$ and $S_{ij}$ has a null eigenvector $\mu$ at some point $x_0\in \Sigma_{t_0}$, i.e., $S_{ij}\mu^j=0$ at $(x_0,t_0)$. If we choose normal coordinates at $(x_0,t_0)$ such that the shifted Weingarten matrix is diagonalized with eigenvalues $\kappa=(\kappa_1,\ldots,\kappa_n)$ in increasing order, then the null vector $\mu$ is the eigenvector $e_n$ corresponding to the eigenvalue $\kappa_n$. Hence, at $(x_0,t_0)$, we have $\kappa_n=\epsilon F$.
We first look at the zero order terms of  \eqref{s6:1-a1} at $(x_0,t_0)$. The terms of the second line  in \eqref{s6:1-a1} satisfy the null eigenvector condition, so can be ignored.
The terms in the third line  of \eqref{s6:1-a1} are nonnegative at $(x_0,t_0)$, since
\begin{align*}
\epsilon F^2-\dot{F}^{kl}(\hat{h}^2)_{kl}\geq &~\epsilon F^2-\kappa_n \dot{F}^{kl}\hat{h}_{kl}=\epsilon F^2-\epsilon F \dot{F}^{kl}\hat{h}_{kl}=0.
\end{align*}
The concavity of $F$ implies that $\dot{F}^{kl}g_{kl}\geq n$. Then the last line of \eqref{s6:1-a1} is also nonnegative at $(x_0,t_0)$ since
\begin{align*}
 \epsilon^2F^{2}\dot{F}^{kl}g_{kl}-\dot{F}^{kl}(\hat{h}^2)_{kl}\geq  & \epsilon(\epsilon n-1)F^2\geq 0.
\end{align*}
To apply Theorem \ref{s2:tensor-mp}, we still need to show that
\begin{align}\label{s6:1-a2}
 0\leq Q_1:= & (p+1)F^{-1}\left(\nabla_nF\nabla_nF-\epsilon \dot{F}^{kl}\nabla_kF\nabla_lF\right)-\ddot{F}^{kl,rs}\nabla_n\hat{h}_{kl}\nabla_n\hat{h}_{rs}\nonumber\\
 &\quad +2\sup_{\Gamma}\dot{F}^{kl}(2\Gamma_{k}^{i}\nabla_{l}S_{ij}\mu^{j}-\Gamma_{k}^{i}\Gamma_{l}^{j}S_{ij}).
\end{align}
Without loss of generality, we can assume that $\kappa_1<\kappa_2<\cdots<\kappa_n$  at $(x_0,t_0)$ (cf. \cite{BAndrews07}).
At $(x_0,t_0)$, we have $S_{nn}=\epsilon F-\kappa_n=0,\ \nabla_kS_{nn}=\epsilon\nabla_kF-\nabla_k\hat{h}_{nn}=0$. Taking $\Gamma_k^i=\frac{\nabla_k S_{ni}}{S_{ii}}$, the supremum over $\Gamma$ in \eqref{s6:1-a2} can be computed as follows.
\begin{align*}
2\sup_{\Gamma}\dot{F}^{kl}(2\Gamma_{k}^{i}\nabla_{l}S_{ij}\mu^{j}-\Gamma_{k}^{i}\Gamma_{l}^{j}S_{ij})=&2\sum_{k=1}^{n}\sum_{i=1}^{n-1}\dot{f}^k\frac{(\nabla_kS_{ni})^2}{S_{ii}}
=2\sum_{k=1}^{n}\sum_{i=1}^{n-1}\dot{f}^k\frac{(\nabla_n\hat{h}_{ki})^2}{\kappa_n-\kappa_i}.
\end{align*}
Using \eqref{s2:F-ddt} to rewrite the terms involving second derivatives of $F$ in \eqref{s6:1-a2}, we have
\begin{align}
Q_1= &(p+1)\kappa_n^{-2}\left(F(\nabla_n\hat{h}_{nn})^2-\kappa_n\sum_{i=1}^n\dot{f}^i(\nabla_i\hat{h}_{nn})^2\right)\nonumber\\
&-\left(\ddot{f}^{ij}\nabla_n\hat{h}_{ii}\nabla_n\hat{h}_{jj}+2\sum_{i>j}\frac{\dot{f}^i-\dot{f}^j}{\kappa_i-\kappa_j}(\nabla_n\hat{h}_{ij})^2\right)+2\sum_{k=1}^{n}\sum_{i=1}^{n-1}
\dot{f}^k\frac{(\nabla_n\hat{h}_{ki})^2}{\kappa_n-\kappa_i}\nonumber\\
\geq&(p+1)\kappa_n^{-2}\left(F-\dot{f}^n\kappa_n\right)(\nabla_n\hat{h}_{nn})^2\nonumber\\
&+\sum_{i<n}\left(-2\frac{\dot{f}^n-\dot{f}^i}{\kappa_n-\kappa_i}-\frac{(p+1)}{\kappa_n}\dot{f}^i+2\frac{\dot{f}^n}{\kappa_n-\kappa_i}\right)(\nabla_i\hat{h}_{nn})^2\nonumber\\
=&(p+1)\kappa_n^{-2}\sum_{i=1}^{n-1}\dot{f}^i\kappa_i(\nabla_n\hat{h}_{nn})^2+\sum_{i<n}\dot{f}^i\left(\frac{2}{\kappa_n-\kappa_i}-\frac{p+1}{\kappa_n}\right)(\nabla_i\hat{h}_{nn})^2\nonumber\\
\geq& (1-p)\sum_{i<n}\dot{f}^i\kappa_n^{-1}(\nabla_i\hat{h}_{nn})^2\geq ~0\nonumber,
\end{align}
where we used the properties for concave function in Lemma \ref{s2:lem-conc} and $p\leq 1$. The tensor maximum principle implies that $S_{ij}\geq 0$ is preserved along flow \eqref{eq-flow} and therefore $\kappa_n\leq\epsilon F$.  The assumption that $F$ approaches zero on the boundary of $\Gamma_+$ implies that $\kappa_n\leq C\kappa_1$ for some positive constant $C$ depending on $\Sigma_0$. In fact, if this is not true, then there exists a sequence of $\kappa^i=(\kappa_1^i,\cdots,\kappa_n^i)$ in $\Gamma_+$ with $\kappa_n^i\leq \epsilon f(\kappa^i)$ and $\kappa_n^i\geq i\kappa_1^i$. Let $\tilde{\kappa}^i=\kappa^i/{\kappa_n^i}$. Then $\tilde{\kappa}_n^i=1, f(\tilde{\kappa}^i)\geq 1/{\epsilon}$ and $\tilde{\kappa}^i_1\leq 1/i$. Since the set $\{\kappa\in \overline{\Gamma}_+:~\kappa_n\leq 1\}$ is compact, we can find a subsequence $\tilde{\kappa}^{i'}$ of $\tilde{\kappa}^i$ converging to a limit $\kappa$ with $\kappa_n= 1, f(\kappa)\geq 1/{\epsilon}$ and $\kappa_1=0$. This contradicts the assumption that $f$ is zero on the boundary of $\Gamma_+$.

(b). $F$ is concave and inverse concave.  Consider the tensor $S_{ij}:=\hat{h}_{ij}-\epsilon\hat{H}g_{ij}$, where $0<\epsilon\leq 1/n$ is chosen such that $S_{ij}$ is positive definite initially.  We aim to prove that $S_{ij}\geq 0$ is preserved. Taking the trace of \eqref{s2:eq-sff}, we have
\begin{align}\label{s6:H}
  \left(\frac{\partial }{\partial t} -\mathcal{L}\right) \hat{H}= & \sum_{m=1}^n\ddot{\Phi}^{kl,rs}\nabla_{m}\hat{h}_{kl}\nabla^{m}\hat{h}_{rs}+\left(\dot{\Phi}^{kl}(\hat{h}^2)_{kl}+2\Phi\right)\hat{H} \nonumber\\
   &\quad +n\dot{\Phi}^{kl}(\hat{h}^2)_{kl}+\left((p+1)\Phi-\dot{\Phi}^{kl}g_{kl}\right)|\hat{A}|^2,
\end{align}
where $|\hat{A}|^2=\sum_{i,j}\hat{h}_i^j\hat{h}_j^i$.  By the equations \eqref{s3:1}, \eqref{s3:1-ssff} and \eqref{s6:H}, we calculate that $S_{ij}$ evolves by
\begin{align}\label{s6:1-1}
\left(\frac{\partial }{\partial t} -\mathcal{L}\right)S_{ij}=&\left(\frac{\partial }{\partial t} -\mathcal{L}\right) \hat{h}_{ij}-\epsilon\left(\frac{\partial }{\partial t} -\mathcal{L}\right) \hat{H}g_{ij}-\epsilon \hat{H}\frac{\partial }{\partial t}g_{ij}\nonumber\\
=&\ddot{\Phi}^{kl,rs}\nabla_{i}\hat{h}_{kl}\nabla_{j}\hat{h}_{rs}-\epsilon g_{ij}\sum_{m=1}^n\ddot{\Phi}^{kl,rs}\nabla_{m}\hat{h}_{kl}\nabla^{m}\hat{h}_{rs}+\dot{\Phi}^{kl}(\hat{h}^2)_{kl}\left(\hat{h}_{ij}+g_{ij}\right)\nonumber\\
&+\left((p-1)\Phi-\dot{\Phi}^{kl}g_{kl}\right)(\hat{h}^2)_{ij}-\epsilon \left(\dot{\Phi}^{kl}(\hat{h}^2)_{kl}+2\Phi\right)\hat{H}g_{ij}-n\epsilon \dot{\Phi}^{kl}(\hat{h}^2)_{kl}g_{ij}\nonumber\\
&-\epsilon\left((p+1)\Phi-\dot{\Phi}^{kl}g_{kl}\right)|\hat{A}|^2g_{ij}+2\epsilon \hat{H}\Phi h_{ij}\nonumber\\
=& \ddot{\Phi}^{kl,rs}\nabla_{i}\hat{h}_{kl}\nabla_{j}\hat{h}_{rs}-\epsilon g_{ij}\sum_{m=1}^n\ddot{\Phi}^{kl,rs}\nabla_{m}\hat{h}_{kl}\nabla^{m}\hat{h}_{rs}\nonumber\\
& +\left((p-1)\Phi-\dot{\Phi}^{kl}g_{kl}\right)({S}^2)_{ij}+\left(\dot{\Phi}^{kl}(\hat{h}^2)_{kl}+2\epsilon \hat{H}(p\Phi-\dot{\Phi}^{kl}g_{kl})\right)S_{ij}\nonumber\\
& +\dot{\Phi}^{kl}(\hat{h}^2)_{kl}(1-\epsilon n)g_{ij}+\left((p+1)\Phi-\dot{\Phi}^{kl}g_{kl}\right)\epsilon(\epsilon \hat{H}^2-|\hat{A}|^2)g_{ij}.
\end{align}
Suppose that there exists a time $t_0\in [0,T^*)$ such that $S_{ij}\geq0$ for all $t\in [0,t_0]$, and  $S_{ij}\mu^{j}=0$ at some point $x_0\in\Sigma_{t_0}$ in the direction $\mu$.
Similar to Case (a), if we choose normal coordinates at $(x_0,t_0)$ such that the shifted Weingarten matrix is diagonalized with eigenvalues $\kappa=(\kappa_1,\ldots,\kappa_n)$ in increasing order, then the null vector $\mu$ is the eigenvector $e_1$ corresponding to the eigenvalue $\kappa_1$. The terms in the second line of \eqref{s6:1-1} satisfy the null eigenvector condition at $(x_0,t_0)$, so can be ignored.  The third line is nonnegative since $\epsilon\leq {1}/{n}$ and $|\hat{A}|^2\geq \hat{H}^2/n$.

By assumption that $F$ is concave and inverse concave, we have that $\Phi=-F^{-p}$ is concave for all $p>0$. Restricting to $0<p\leq1$, we can also show that $\Phi=-F^{-p}$ is inverse concave in the sense of the definition in Andrews' paper \cite{BAndrews07}, i.e., $\Phi^*(x_1,\cdots,x_n)=-\Phi(x_1^{-1},\cdots,x_n^{-1})$ is concave. Therefore, we can apply Theorem 4.1 in \cite{BAndrews07} to conclude that
\begin{equation}\label{ben-07}
\ddot{\Phi}^{kl,rs}\nabla_{1}\hat{h}_{kl}\nabla_{1}\hat{h}_{rs}-\epsilon\sum_{i=1}^n\ddot{\Phi}^{kl,rs}\nabla_{i}\hat{h}_{kl}\nabla^{i}\hat{h}_{rs}+2\sup_{\Gamma}\dot{\Phi}^{kl}(2\Gamma_{k}^{i}\nabla_{l}S_{ij}\mu^{j}-\Gamma_{k}^{i}\Gamma_{l}^{j}S_{ij})\geq0.
\end{equation}
Hence $S_{ij}\geq 0$ is preserved by applying the tensor maximum principle in Theorem \ref{s2:tensor-mp}. This implies the estimate \eqref{est-pinch-p<1} immediately.

(c). $F$ is inverse concave and $F_*|_{\partial\Gamma_+}=0$. We consider the tensor $S_{ij}:=\hat{h}_{ij}-\epsilon Fg_{ij}$, where $0<\epsilon\leq 1/n$ is chosen such that $S_{ij}$ is positive definite initially.   The calculation given in case (a) shows that
\begin{align}\label{h-eF}
\left(\frac{\partial }{\partial t} -\mathcal{L}\right)S_{ij}=&pF^{-p-1}\biggl(\ddot{F}^{kl,rs}\nabla_i\hat{h}_{kl}\nabla_j\hat{h}_{rs}-(p+1)F^{-1}\left(\nabla_iF\nabla_jF-\epsilon g_{ij}\dot{F}^{kl}\nabla_kF\nabla_lF\right)\biggr) \nonumber \\
&-\left(\dot{\Phi}^{kl}g_{kl}-(p-1)\Phi\right)(S^2)_{ij}+\left(\dot{\Phi}^{kl}(\hat{h}^2)_{kl}+2\epsilon F(p\Phi-\dot{\Phi}^{kl}g_{kl})\right)S_{ij}\nonumber\\
&+(p+1)\epsilon F^{-p}\left(\dot{F}^{kl}(\hat{h}^2)_{kl}-\epsilon F^2\right)g_{ij}\nonumber\\
&+pF^{-p-1}\left(\dot{F}^{kl}(\hat{h}^2)_{kl}-\epsilon^2F^{2}\dot{F}^{kl}g_{kl}\right)g_{ij}.
\end{align}
Suppose that there exists a time $t_0\in [0,T^*)$ such that $S_{ij}\geq0$ for all $t\in [0,t_0]$ and $S_{ij}\mu^{j}=0$ at some point $x_0\in \Sigma_{t_0}$ in the direction $\mu$.
Similar to Case (b), the null vector $\mu$ is the eigenvector $e_1$ corresponding to the smallest eigenvalue $\kappa_1$ of the shifted Weingarten Matrix.
 The terms in the second line of \eqref{h-eF} satisfy the null eigenvector condition at $(x_0,t_0)$, so can be ignored.  The terms in the third line and the last line of  \eqref{h-eF} are nonnegative at $(x_0,t_0)$ since we have $\kappa_1=\epsilon F$ at this point. To apply Theorem \ref{s2:tensor-mp}, we need to prove the following inequality
\begin{align}\label{s6:1-c1}
0\leq Q_1:=&\ddot{F}^{kl,rs}\nabla_1\hat{h}_{kl}\nabla_1\hat{h}_{rs}-(p+1)F^{-1}\left((\nabla_1F)^2-\epsilon \dot{F}^{kl}\nabla_kF\nabla_lF\right)\nonumber \\
&+2\sup_{\Gamma}\dot{F}^{kl}\left(2\Gamma_{k}^{i}\nabla_{l}S_{ij}\mu^{j}-\Gamma_{k}^{i}\Gamma_{l}^{j}S_{ij}\right).
\end{align}
The proof of \eqref{s6:1-c1} is similar to the pinching estimate in \cite[pp.1563-1564]{YW}, by using \eqref{s2:F-ddt}, $p\leq 1$ and the property of inverse concave functions in Lemma \ref{inv-concave}. We omit the details here. Once we have that $S_{ij}\geq 0$ is preserved, we can use the assumption that $f_*$ vanishes at the boundary of $\Gamma_+$ and the argument given in the end of case (a) to conclude the pinching estimate \eqref{est-pinch-p<1}.

(d). $n=2$. In this case, we don't need any second derivative conditions on $F$. We consider the following function
\begin{equation}\label{G-n=2}
G:=\left(\frac{\kappa_1-\kappa_2}{\kappa_1+\kappa_2}\right)^2,
\end{equation}
which is homogeneous of degree zero of the shifted principal curvatures $\kappa_1,\ \kappa_2$. By \eqref{s2:eq-sff} we have
\begin{align}\label{case-c-G}
(\P_t-\mathcal{L})G=&\left(\dot{G}^{ij}\ddot{\Phi}^{kl,rs}-\dot{\Phi}^{ij}\ddot{G}^{kl,rs}\right)\nabla_i\hat{h}_{kl}\nabla_j\hat{h}_{rs}+\dot{\Phi}^{kl}(\hat{h}^2)_{kl}\dot{G}^{ij}g_{ij}\nonumber\\
&+((p+1)\Phi-\dot{\Phi}^{kl}g_{kl})\dot{G}^{ij}(\hat{h}^2)_{ij}.
\end{align}
Here we view $G$ as a symmetric function of the shifted Weingarten matrix. The zero order terms of \eqref{case-c-G} can be estimated as follows.
\begin{equation}\label{case-c-Q0}
Q_0=-\frac{4G}{\kappa_1+\kappa_2}\dot{\Phi}^{kl}(\hat{h}^2)_{kl}+\frac{4\kappa_1\kappa_2G}{\kappa_1+\kappa_2}((p+1)\Phi-\dot{\Phi}^{kl}g_{kl})\leq 0.
\end{equation}
The same argument in \cite[\S 3]{L-W-W} gives that the gradient terms of \eqref{case-c-G} are non-positive at the critical point of $G$ when $0<p\leq1$. Then the maximum principle guarantees that $G$ is uniformly bounded along the flow and the pinching estimate \eqref{est-pinch-p<1} follows immediately.
\end{proof}

\subsection{Maximal existence}$\ $

The pinching estimate implies that $0<1/C\leq \dot{f}^k\leq C$ for all $k=1,\cdots,n$. In particular,
\begin{equation}\label{s6:dotf}
  \sum_{k=1}^n\dot{f}^k\geq C>0,~\dot{f}^{k}\kappa_k^2\leq CF^2
\end{equation}
for some uniform positive constant $C$.
 On the other hand, applying maximum principle to \eqref{s6:F} we have a uniform upper bound on $F$. Hence, $\dot{f}^{k}\kappa_k^2\leq CF^2\leq \tilde{C}F$. By Remark \ref{s5:rem1}, we have that if the graph function $u$ of the graph $\Sigma_t=\mathrm{graph} ~u$ is bounded, i.e., $u\leq \bar{r}$, then the speed function $F\geq C(p,\bar{r},\Sigma_0)>0$. Combining these facts with the pinching estimate in Lemma \ref{pinch-p<1}, we derive uniform estimate on the shifted principal curvatures
\begin{equation}\label{s6:2-1}
  0< \frac 1C\leq \kappa_i\leq C
\end{equation}
for some positive constant $C$ depending on the upper bound $\bar{r}$ of $u$, $p$ and  $\Sigma_0$. The estimate \eqref{s6:2-1} together with Lemma \ref{est-C1} and $\inf_{\Sigma_0}u\leq u\leq \bar{r}$ implies the uniform $C^2$ bound on $u$. To derive the $C^{2,\alpha}$ estimate, we apply Krylov-Safonov's estimate \cite{Kry} for cases (a) and (b), and Andrews' estimate \cite{And04} for case (d).

For case (c), the function $F$ is inverse concave, we use the horospherical support function of horo-convex hypersurfaces recalled in \S \ref{sec:2-2} to convert the equation \eqref{eq-flow} to a parabolic equation with concave elliptic part. Recall the identity \eqref{s6:2-s1} and the definition of the matrix $(A_{ij}[s])$ given in \eqref{s6:2-sA}. By a similar argument to that in Proposition 5.3 of \cite{ACW18}, the flow \eqref{eq-flow} is equivalent to a parabolic equation of the horospherical support function $s$ on $\mathbb{S}^n$:
\begin{equation}\label{s6:2-s2}
  \frac{\partial }{\partial t}s=e^{ps}F^p_*(A_i^j[s]).
\end{equation}
Note that if we write $\Sigma=\mathrm{graph}~u$ as radial graph of a function $u$ on $\mathbb{S}^n$, for each $z\in \mathbb{S}^n$, $u(z)$ is just the hyperbolic geodesic distance of the point $X(z)=(u(z),z)\in \mathbb{R}_+\times \mathbb{S}^{n}$ to  the center of the geodesic polar coordinates. Therefore, the function $s$ differs from $u$ by a diffeomorphism of $\mathbb{S}^n$. In fact, from \eqref{s6:2-s} we have
\begin{equation*}
  \cosh u=\frac 12e^s|\bar{\nabla}s|^2+\cosh s.
\end{equation*}
Then the bound $\inf_{\Sigma_0}u\leq u\leq \bar{r}$ implies $C^0$ and $C^1$ estimates of $s$. This together with the estimate \eqref{s6:2-1}, and \eqref{s6:2-s1}, \eqref{s6:2-sA} implies uniform $C^2$ bound on $s$. Since $F_*$ is concave and $p\leq 1$, $F_*^p$ is also a concave function. Then the equation \eqref{s6:2-s2} is uniformly parabolic and is concave with respect to the second spatial derivatives of $s$. The Krylov-Safonov's estimate \cite{Kry} can be applied to derive the $C^{2,\alpha}$ estimate. The higher order estimates for all cases follow from the parabolic Schauder estimate. Therefore, similar to the proof of Proposition \ref{s5:prop1}, we have
\begin{prop}
	Under the assumption of Theorem \ref{thm-0<p<1}, the solution $\Sigma_t$ of flow \eqref{eq-flow} expands smoothly to infinity in finite time and the graph function of the solution $\Sigma_t$ satisfies
	\begin{equation*}
	\limsup_{t\to T^*}\max u(\cdot,t)=\infty.
	\end{equation*}
\end{prop}
With the estimate \eqref{s6:dotf} in hand, we can also apply the arguments in Lemma \ref{est-FQ-lower} and Lemma \ref{est-FQ-upper}, and use the pinching estimate in Lemma \ref{pinch-p<1} to prove uniform bounds on rescaled curvature quantities $FQ$ and $\kappa_iQ$.
\begin{lem}\label{est-FQ}
Under the assumption of Theorem \ref{thm-0<p<1}, we have
\begin{equation}\label{s6:1-lem2-1}
nC^{-1}\leq FQ\leq nC,
\end{equation}
and
\begin{equation}\label{s6:1-lem2-2}
C^{-1}\leq \kappa_iQ\leq C
\end{equation}
along flow \eqref{eq-flow}, where $C$ is a positive constant depending only on $p,\Sigma_0$ .
\end{lem}
Then we have
\begin{cor}
	Under the assumpution of Theorem \ref{thm-0<p<1}, for any $m\in \mathbb{N}$, we have $ |u-\theta|_{C^m(\mathbb{S}^n)}\leq c_m$ for positive constants $c_m$ depending on $p,\Sigma_0$.
\end{cor}
To prove the above corollary, we can modify  the argument in the end of Section \ref{sec:thm1}. As before, the only difference is the $C^{2,\alpha}$ estimate,  which can be proved by using Krylov-Safonov's theorem \cite{Kry} for cases (a)-(b), and  Andrews' estimate \cite{And04} for case (d). For case (c), we consider estimating $s(z,t)-\theta(t)$ instead:  The equation \eqref{s6:2-s2} together with \eqref{s5:tau} implies that with respect to the new time parameter $\tau=\tau(t)$, $s-\theta$ satisfies
\begin{equation}\label{s6:2-stau}
  \frac{\partial}{\partial\tau}(s-\theta)=F_*^p(e^sA_i^jQ^{-1})-\frac 1{n^p}.
\end{equation}
Lemma \ref{est-FQ} implies that the right hand side of \eqref{s6:2-stau} is uniformly parabolic and is concave with respect to the second spatial derivatives of $s$. Thus the estimate of Krylov-Safonov can be applied to derive $C^{2,\alpha}$ estimate.  The higher order estimate follows from the standard parabolic Schauder estimate.

\subsection{Improved pinching estimates}$\ $

In previous two subsections, we have shown that the pinching ratio of the evolving hypersurface $\Sigma_t$ is bounded by its initial value and $\Sigma_t$ expands to infinity in finite time. Now we improve the pinching estimate as below, which says that the pinching ratio tends to $1$ as $t\to T^*$.
\begin{lem}\label{est-umbilic}
Under the assumptions in Theorem \ref{thm-0<p<1}, there exist two positive constants $C,\ \delta$ such that
\begin{equation}\label{s6:2-lem1}
\F{\kappa_n}{\kappa_1}-1\leq CQ(t)^{-\delta},
\end{equation}	
where $\kappa_1\leq\kappa_2\dots\leq\kappa_n$ are the shifted principal curvatures of $\Sigma_t$.
\end{lem}
\begin{proof}
As in Lemma \ref{pinch-p<1}, we consider the four cases of $F$ separately.

(a). $F$ is concave and $F|_{\partial\Gamma_+}=0$. We consider $S_{ij}:=\epsilon(t)Fg_{ij}-\hat{h}_{ij}$, where $\epsilon(t)=\frac{1}{n}+mQ(t)^{-\delta}$, and $m,\ \delta$ are chosen such that $S_{ij}$ is positive definite initially. By the calculation in case (a) of Lemma \ref{pinch-p<1} and using \eqref{Q},  we have
\begin{align}\label{s6:2-a1}
\left(\frac{\partial }{\partial t} -\mathcal{L}\right)S_{ij}=&pF^{-p-1}\biggl((p+1)F^{-1}\left(\nabla_iF\nabla_jF-\epsilon g_{ij}\dot{F}^{kl}\nabla_kF\nabla_lF\right)-\ddot{F}^{kl,rs}\nabla_i\hat{h}_{kl}\nabla_j\hat{h}_{rs}\biggr) \nonumber \\
&+\left(\dot{\Phi}^{kl}g_{kl}-(p-1)\Phi\right)(S^2)_{ij}+\left(\dot{\Phi}^{kl}(\hat{h}^2)_{kl}+2\epsilon F(p\Phi-\dot{\Phi}^{kl}g_{kl})\right)S_{ij}\nonumber\\
&+(p+1)\epsilon F^{-p}\left(\epsilon F^2-\dot{F}^{kl}(\hat{h}^2)_{kl}\right)g_{ij}\nonumber\\
&+pF^{-p-1}\left(\epsilon^2F^{2}\dot{F}^{kl}g_{kl}-\dot{F}^{kl}(\hat{h}^2)_{kl}\right)g_{ij}-m\delta n^{-p}(2Q+1)Q^{p-\delta-1}Fg_{ij}.
\end{align}
Suppose that there exists a time $t_0\in [0,T^*)$ such that $S_{ij}\geq0$ for all $t\in [0,t_0]$ and $S_{ij}\mu^{j}=0$ at some point $x_0\in \Sigma_{t_0}$ in the direction $\mu$. As in Lemma \ref{pinch-p<1},  the null vector $\mu$ is the eigenvector $e_n$ corresponding to the largest eigenvalue $\kappa_n$ of the shifted Weingarten Matrix. At $(x_0,t_0)$, the second line of \eqref{s6:2-a1} vanishes  and the third line is nonnegative. The last line of \eqref{s6:2-a1} can be estimated as
\begin{align*}
  & pF^{-p-1}\left(\epsilon^2F^{2}\dot{F}^{kl}g_{kl}-\dot{F}^{kl}(\hat{h}^2)_{kl}\right)g_{ij}-m\delta n^{-p}(2Q+1)Q^{p-\delta-1}Fg_{ij} \\
  \geq  & pF^{-p-1}\epsilon(\epsilon n-1)F^2g_{ij}-m\delta n^{-p}(2Q+1)Q^{p-\delta-1}Fg_{ij}\\
  =& \biggl(pF^{-p-1}\epsilon mnQ^{-\delta}F^2-m\delta n^{-p}(2Q+1)Q^{p-\delta-1}F\biggr)g_{ij}\\
  \geq &\biggl(pmCQ^{p-1-\delta}-m\delta CQ^{p-2-\delta}(2Q+1)\biggr)g_{ij}~\geq~0
\end{align*}
by choosing $\delta$ smaller (depending only on $p,\Sigma_0$), where we used the estimate $nC^{-1}\leq FQ\leq nC$ in Lemma \ref{est-FQ}. With the estimate \eqref{s6:1-a2} on the gradient terms, we can apply Theorem \ref{s2:tensor-mp} to conclude that $S_{ij}\geq 0$ is preserved along flow \eqref{eq-flow}. Then by the second inequality of \eqref{s2:conc} due to the concavity of $F$, we have
\begin{equation*}
\kappa_n\leq(\frac{1}{n}+mQ^{-\delta})F\leq(\frac{1}{n}+mQ^{-\delta})\sum_{i=1}^n\kappa_i.
\end{equation*}
This implies that
\begin{align*}
\F{\kappa_n}{\kappa_1}-1=&\F{\kappa_n-\kappa_1}{\kappa_1}\leq\F{n\kappa_n-\sum_{i=1}^n\kappa_i}{\kappa_1}\nonumber\\
\leq & nmQ^{-\delta}\sum_{i=1}^n\frac{\kappa_i}{\kappa_1}~\leq ~ nmCQ^{-\delta},
\end{align*}
where we used the pinching estimate \eqref{est-pinch-p<1}  in the last inequality.

(b). $F$ is concave and inverse concave. Consider $S_{ij}:=\hat{h}_{ij}-\epsilon(t)\hat{H}g_{ij}$, where $\epsilon(t)=\frac{1}{n}-mQ(t)^{-\delta}$, and $m,\ \delta$ are chosen such that $S_{ij}$ is positive definite initially. By \eqref{Q} and \eqref{s6:1-1}, we have
\begin{align}\label{s6:2-b1}
\left(\frac{\partial }{\partial t} -\mathcal{L}\right)S_{ij}=&\ddot{\Phi}^{kl,rs}\nabla_{i}\hat{h}_{kl}\nabla_{j}\hat{h}_{rs}-\epsilon g_{ij}\sum_{m=1}^n\ddot{\Phi}^{kl,rs}\nabla_{m}\hat{h}_{kl}\nabla_{m}\hat{h}_{rs}\nonumber\\
& +\left((p-1)\Phi-\dot{\Phi}^{kl}g_{kl}\right)({S}^2)_{ij}+\left(\dot{\Phi}^{kl}(\hat{h}^2)_{kl}+2\epsilon \hat{H}(p\Phi-\dot{\Phi}^{kl}g_{kl})\right)S_{ij}\nonumber\\
& +\dot{\Phi}^{kl}(\hat{h}^2)_{kl}(1-\epsilon n)g_{ij}+\left((p+1)\Phi-\dot{\Phi}^{kl}g_{kl}\right)\epsilon(\epsilon \hat{H}^2-|\hat{A}|^2)g_{ij}\nonumber\\
&-m\delta n^{-p}(2Q+1)Q^{p-1-\delta}\hat{H}g_{ij}.
\end{align}
Suppose that there exists a time $t_0\in [0,T^*)$ such that $S_{ij}\geq0$ for all $t\in [0,t_0]$,  and $S_{ij}\mu^{j}=0$ at some point $x_0\in\Sigma_{t_0}$ in the direction $\mu$. As in Lemma \ref{pinch-p<1},  the null vector $\mu$ is the eigenvector $e_1$ corresponding to the smallest eigenvalue $\kappa_1$ of the shifted Weingarten Matrix. We denote the zero order terms in \eqref{s6:2-b1} by $Q_0$. Then at $(x_0,t_0)$, we have
\begin{align*}
  Q_0\mu_i\mu^j\geq  &\dot{\Phi}^{kl}(\hat{h}^2)_{kl}(1-\epsilon n)-m\delta n^{-p}Q^{p-1-\delta}(2Q+1)\hat{H}\\
   \geq& mnp\epsilon F^{-p}\hat{H}Q^{-\delta}-m\delta n^{-p}Q^{p-1-\delta}(2Q+1)\hat{H}\\
   \geq &mQ^{p-\delta}\hat{H} \left(np\epsilon C -\delta n^{-p}(2+Q^{-1})\right)~\geq~0
\end{align*}
by possibly choosing $\delta$ smaller (depending only on $n,p,\Sigma_0$), where we used \eqref{s6:1-lem2-1}  in the third inequality. Recall that we have \eqref{ben-07}, we can apply Theorem \ref{s2:tensor-mp} to conclude that $S_{ij}\geq 0$ is preserved along flow \eqref{eq-flow}. The estimate \eqref{s6:2-lem1} follows immediately.

(c). $F$ is inverse concave and $F_*|_{\partial\Gamma_+}=0$. The proof is similar to cases (a) and (b), by considering  $S_{ij}:=\hat{h}_{ij}-\epsilon(t)Fg_{ij}$, where $\epsilon(t)=\frac{1}{n}-mQ(t)^{-\delta}$ for certain constants $m,\ \delta$, and using the argument in case (c) of Lemma \ref{pinch-p<1} to conclude that $S_{ij}\geq 0$ is preserved along flow \eqref{eq-flow}. Using the concavity of $F_*$ and the second inequality of \eqref{s2:conc}, we have
\begin{equation*}
  f(\kappa)=f_*(\kappa^{-1})^{-1}\geq n^2\left(\sum_{i=1}^n\frac 1{\kappa_i}\right)^{-1}.
\end{equation*}
Then the estimate \eqref{s6:2-lem1} follows from the positivity of $S_{ij}$.

(d). $n=2$. We consider $\tilde{G}:=GQ(t)^{2\delta}$, where $G$ is defined in \eqref{G-n=2} and $\delta>0$ is a constant to be determined later. Then by \eqref{case-c-G}, \eqref{case-c-Q0}, \eqref{s6:dotf} and Lemma \ref{est-FQ}, at the maximum point of $\tilde{G}$, we have
\begin{align*}
\frac{\partial}{\partial t}\tilde{G}=&\frac{\partial}{\partial t}GQ^{2\delta}+2\delta Q^{2\delta-1}\frac d{dt}QG\nonumber\\
\leq&Q^{2\delta} G\left(-\frac{4p}{\kappa_1+\kappa_2}\frac{\sum_i\dot{f}^i\kappa_i^2}{F^{p+1}}-\frac{4\kappa_1\kappa_2}{\kappa_1+\kappa_2}\left(\frac{p+1}{F^p}+\frac{p\sum_i\dot{f}^i}{F^{p+1}}\right)+2^{1-p}\delta(2+Q^{-1})Q^p\right)\nonumber\\
\leq&Q^{2\delta}G\left(-CQ^p+2^{1-p}\delta (2+Q^{-1})Q^p\right).
\end{align*}
Choosing $\delta$ small enough, we obtain that $\tilde{G}$ is non-increasing in time. Therefore $\tilde{G}$ is uniformly bounded from above and the pinching estimate \eqref{s6:2-lem1} follows.
\end{proof}

\section{Proof of Theorem \ref{thm-0<p<1}: Oscillation decay}\label{sec:osc}

In this section, we prove that there exists a point $y\in \mathbb{H}^{n+1}$ such that if we write the solution $\Sigma_t$ as graphs of function $u(\cdot,t)$ in the geodesic polar coordinate system centered at $y$, then the oscillation of $u$ converges to zero exponentially.

\subsection{Hausdorff closeness to a sphere}$\ $

We first show that for each time $t$, $\Sigma_t$ is Hausdorff close to a geodesic sphere. To prove this, we use the conformally flat parametrization and consider the corresponding flow in Poinc\'{a}re ball of $\mathbb{R}^{n+1}$, see e.g.,\cite[\S 5]{Ge-2011}. In the Poinc\'{a}re ball model, the hyperbolic space is the unit ball $B_1^{n+1}$ equipped with the conformally flat metric
\begin{align*}
d\bar{s}^2=&\frac{4}{\left(1-r^2\right)^2}(dr^2+r^2g_{\mathbb{S}^n})\nonumber\\
=&e^{2\psi}(dr^2+r^2g_{\mathbb{S}^n}),
\end{align*}
where $r=|x|$ for each point $x\in B_1^{n+1}$. Let
\begin{equation}\label{eq-breve-u}
u=\ln(1+r)-\ln(1-r).
\end{equation}
Then $u$ is the radial distance in hyperbolic space to the origin of the ball $B_1^{n+1}$. We distinguish quantities in $B_{1}^{n+1}\subset\RR^{n+1}$ from those in $\HH^{n+1}$ by an additional br\`eve, e.g., $\breve{u},\ \breve{h}_{i}^j,\breve{\lambda}_i$.   For a graphical hypersurface
\begin{equation*}
\Sigma=\text{graph}\ u=\text{graph}\ \breve{u},
\end{equation*}
by using \eqref{eq-breve-u}, we have
\begin{equation}\label{s7:1-r}
  \breve{u}=r=1-\frac 2{e^{u}+1},\qquad e^{\psi}=\frac{(e^u+1)^2}{2e^u}.
\end{equation}
The Weingarten matrixes $h_i^j$ and $\breve{h}_i^j$ of the hypersurface $\Sigma$ are related by
\begin{equation}\label{s7:1-psi}
e^{\psi}h_i^j=\breve{h}_i^j+\frac{1}{v}\frac{2r}{1-r^2}\delta_i^j,
\end{equation}
where $v$ is defined by \eqref{s2:v-def}.
The equations \eqref{s7:1-psi} and \eqref{s7:1-r} imply that the principal curvatures $\lambda_i$ and $\breve{\lambda}_i$ satisfy
\begin{align*}
\breve{\lambda}_i=&e^{\psi}\lambda_i-\frac{1}{v}e^{\psi}(1-\frac{2}{e^u+1})\\
=&e^{\psi}\left(\kappa_i+(1-\frac{1}{v})+\frac{2}{(e^u+1)v}\right).
\end{align*}

We have shown that the evolving hypersurface $\Sigma_t$ satisfies the estimate $C^{-1}\leq\kappa_iQ\leq C$ (see \eqref{s6:1-lem2-2}), $v-1\leq C Q^{-1}$ (see \eqref{s4:C1}). Using the expression \eqref{s7:1-r} for $e^{\psi}$, we obtain that
\begin{equation*}
  \breve{\lambda}_i=1+O(Q(t)^{-1/2}).
\end{equation*}
Therefore by Lemma \ref{est-umbilic}, we have
\begin{align}\label{est-umbilic2}
\left|\frac{1}{\breve{\lambda}_i}-\frac{1}{\breve{\lambda}_j}\right|=\left|\frac{\breve{\lambda}_i-\breve{\lambda}_j}{\breve{\lambda}_i\breve{\lambda}_j}\right|=e^{\psi}\left|\frac{\kappa_i-\kappa_j}{\breve{\lambda}_i\breve{\lambda}_j}\right|\leq CQ(t)^{-\F{1}{2}-\delta}
\end{align}
for some positive constant $C=C(p,\Sigma_{0})$.

In Theorem 1.4 of \cite{KL-1999}, Leichtwei\ss~proved that there exists a constant $c_n$ depending only on the dimension $n$ such that for any strictly convex hypersurface $\Sigma$ of $\RR^{n+1}$, there exists a sphere $S$ in $\RR^{n+1}$ such that
\begin{equation*}
\breve{d}_{\mathcal{H}}(\Sigma,S)\leq c_n\max_{x\in\Sigma}(\breve{\mathfrak{r}}_n(x)-\breve{\mathfrak{r}}_1(x)),
\end{equation*}
where $\breve{\mathfrak{r}}_1\leq \breve{\mathfrak{r}}_2\leq\dots\leq \breve{\mathfrak{r}}_n$ are principal radii of curvature of $\Sigma$ and $\breve{d}_{\mathcal{H}}$ is the Euclidean Hausdorff distance. The estimate \eqref{est-umbilic2} says that $\breve{\mathfrak{r}}_n(t)-\breve{\mathfrak{r}}_1(t)\leq C(n,p,\Sigma_{0})Q(t)^{-\frac{1}{2}-\delta}$, where $\breve{\mathfrak{r}}_i=1/\breve{\lambda}_i$ are the principal radii of curvature of $\Sigma_t$ in $B_1^{n+1}\subset\RR^{n+1}$ with respect to Euclidean metric. Then there exists a sphere $S_t$ in $B_1^{n+1}$ such that $\breve{d}_{\mathcal{H}}(\Sigma_t,S_t)\leq C(n,p,\Sigma_{0})Q(t)^{-\frac 12-\delta}$. For the corresponding hyperbolic Hausdorff distance $d_{\mathcal{H}}$, we have
$d_{\mathcal{H}}(\Sigma_t,S_t)\leq \max{(e^{\psi})}\breve{d}_{\mathcal{H}}(\Sigma_t,S_t)\leq C(n,p,\Sigma_{0})Q(t)^{-\delta}$.
 We denote by $y_t$ a suitable oscillation minimizing center of $\Sigma_t$ in $\mathbb{H}^{n+1}$, then we obtain the following proposition.

\begin{prop}\label{s7:1-prop-1}
Under the assumption of Theorem \ref{thm-0<p<1}, there exists a positive constant $C$ such that for each $\Sigma_{t}$, there exists a point $y_t\in\HH^{n+1}$ with
\begin{equation}\label{s7:1-osc}
\text{osc}(u_{y_t})\leq C(n,p,\Sigma_{0})Q(t)^{-\delta},
\end{equation}
where $u_{y_t}$ is the graph representation of $\Sigma_{t}$ in the geodesic polar coordinate centered at $y_t$, $\delta$ is determined in Lemma \ref{est-umbilic}.
\end{prop}

\subsection{Oscillation minimizing center}$\ $

We first show the following consequence of the oscillation decay \eqref{s7:1-osc}.
\begin{lem}\label{kappaQ=1}
Under the assumption of Theorem \ref{thm-0<p<1}, for any $0<\epsilon<\delta/2$ there exists a positive constant $C$ such that
\begin{equation}\label{s7:2-lem1}
|\kappa_i Q-1|\leq CQ(t)^{-\epsilon}
\end{equation}
along flow \eqref{eq-flow} for all $t\in [0,T^*)$, where $\delta$ is determined in Lemma \ref{est-umbilic}.
\end{lem}
\begin{proof}
Since the shifted curvature $\kappa_i$ of a hypersurface in  hyperbolic space $\mathbb{H}^{n+1}$ is independent with the choice of the center of geodesic polar coordinates of $\mathbb{H}^{n+1}$, we calculate $\kappa_i$ on $\Sigma_{t}=\mathrm{graph} ~u_{y_t}$ in the geodesic polar coordinate centered at $y_t$. Armed with the estimate \eqref{s7:1-osc}, we can apply the inequality \eqref{s4:2-v} to improve the $C^1$ estimate in Lemma \ref{est-C1} and obtain that
\begin{equation}\label{s7:2-v}
v^2-1=\frac {|\bar{\nabla}u|^2}{\sinh^{2}u}\leq CQ(t)^{-1-\delta}.
\end{equation}
Equivalently, $|\bar{\nabla}u|\leq CQ^{-\delta/2}$. Since $u-\theta$ is bounded in $C^{\infty}(\mathbb{S}^n)$, via interpolation (see e.g.,Lemma 6.1 in \cite{Ge-2011}) we obtain
\begin{equation*}
|\bar{\nabla}^2u|\leq C|\bar{\nabla}u|^{1-\frac{1}{k}}|\bar{\nabla}^{k+1}u|^{\F{1}{k}}\leq CQ(t)^{-(1-\frac{1}{k})\frac{\delta}{2}},
\end{equation*}
where $C$ depends on $k, n, p, \Sigma_0$ but does not depend on the choice of oscillation center. For any $0<\epsilon<{\delta}/{2}$, by choosing $k\geq\F{\delta}{\delta-2\epsilon}$, we have $|\bar{\nabla}^2u|\leq C(\epsilon,n,p,\Sigma_0)Q(t)^{-\epsilon}$. Then from \eqref{s2:h-hat}, we obtain that
\begin{align*}
\hat{h}_i^jQ-\delta_i^j= ~&~(Q\left(\frac{\coth u}v-1\right)-1)\delta_i^j+\frac{\ch u }{v^3\sinh^3u}Qu_iu^j-v^{-1}Qg^{jk}\bar{\nabla}_i\bar{\nabla}_ku=O(Q^{-\epsilon})\delta_i^j,
\end{align*}
then the estimate \eqref{s7:2-lem1} follows immediately.
\end{proof}

Now we prove that the oscillation minimizing center $y_t$ converges to a fixed point $y$ as $t\to T^*$.
\begin{prop}\label{s7:2-osc}
There exists a point $y\in\HH^{n+1}$ such that for any $\ 0<\epsilon<\min\{\delta/2,1\}$, we have
\begin{equation}
\text{osc}(u_y)\leq CQ(t)^{-\epsilon}
\end{equation}
along flow \eqref{eq-flow} for some positive constant $C=C(p,\epsilon,\Sigma_0)$, where $u_y$ is the graph representation of $\Sigma_{t}\subset\HH^{n+1}$ over the geodesic sphere centered at $y$ and $\delta$ is determined in Lemma \ref{est-umbilic}.
\end{prop}
\begin{proof}
 Under the geodesic polar coordinate system with center in the domain enclosed by $\Sigma_0$, we consider the equation \eqref{eq-u-theta} for $u-\theta$ with respect to the time parameter $\tau=\tau(t)$. By Lemma \ref{kappaQ=1} and the estimate \eqref{s4:C1} on $v$, for any $\epsilon<\min\{\delta/2,1\}$ we have
\begin{align*}
\left|\frac{\P}{\P\tau}(u-\theta)\right|= &~\biggl|\frac v{F(\hat{h}_i^jQ)^p}-\frac 1{n^p}\biggr|\leq C_0Q^{-\epsilon}\leq C_1e^{-2\epsilon\theta}.
\end{align*}
This combined with \eqref{s5:dtau-theta} gives that
\begin{equation*}
\frac{\P}{\P\tau}\left(u-\theta-\frac{C_1n^p}{2\epsilon}e^{-2\epsilon\theta}\right)\geq-C_1e^{-2\epsilon\theta}+C_1e^{-2\epsilon\theta}=0,
\end{equation*}
\begin{equation*}
\frac{\P}{\P\tau}\left(u-\theta+\frac{C_1n^p}{2\epsilon}e^{-2\epsilon\theta}\right)\leq C_1e^{-2\epsilon\theta}-C_1e^{-2\epsilon\theta}=0.
\end{equation*}
From the monotonicity above and the fact that $u-\theta$ is $C^\infty$ bounded, we obtain that $u-\theta$ converges in $C^\infty$ to a smooth function on the sphere.
Moreover, for any $\tau>0$, let $y_{\tau}$ be the oscillation minimizing center of $\Sigma_{t(\tau)}$ and denote $u_{y_{\tau}}$ the graph representation of $\Sigma_{t(\tau)}$ in the geodesic polar coordinate centered at $y_{\tau}$. For any $\tau_2>\tau_1$, Proposition \ref{s7:1-prop-1} implies that
\begin{align*}
(u_{y_{\tau_1}}-\theta-\frac{C_1n^p}{2\epsilon}e^{-2\epsilon\theta})|_{\tau_2}\geq & (u_{y_{\tau_1}}-\theta-\frac{C_1n^p}{2\epsilon}e^{-2\epsilon\theta})|_{\tau_1}\geq-cQ^{-\epsilon}|_{\tau_1},\\
(u_{y_{\tau_1}}-\theta+\frac{C_1n^p}{2\epsilon}e^{-2\epsilon\theta})|_{\tau_2}\leq & (u_{y_{\tau_1}}-\theta+\frac{C_1n^p}{2\epsilon}e^{-2\epsilon\theta})|_{\tau_1}\leq cQ^{-\epsilon}|_{\tau_1}.
\end{align*}
Then we have $-cQ(\tau_1)^{-\epsilon}\leq(u_{y_{\tau_1}}-\theta)|_{\tau_2}\leq cQ(\tau_1)^{-\epsilon}$. Therefore
\begin{equation*}
\text{osc}(u_{y_{\tau_1}})(\tau_2)\leq cQ(\tau_1)^{-\epsilon},\qquad \forall\ \tau_2>\tau_1.
\end{equation*}
Note that $\text{osc}(u_{y_{\tau_2}})(\tau_2)\leq cQ(\tau_2)^{-\delta}\leq cQ(\tau_1)^{-\delta}$, then we obtain that
\begin{equation*}
\mathrm{d}_{\mathbb{H}}(y_{\tau_1},y_{\tau_2})\leq cQ(\tau_1)^{-\epsilon},\qquad \forall\ \tau_2>\tau_1.
\end{equation*}
Hence $y_\tau$ converges to a point $y\in\HH^{n+1}$ and $\mathrm{d}_{\mathbb{H}}(y_{\tau},y)\leq cQ(\tau)^{-\epsilon}$ for any $\tau>0$. It follows that
\begin{equation*}
\text{osc}(u_y)\leq\text{osc}(u_{y_{\tau}})+2\mathrm{d}_{\mathbb{H}}(y_{\tau},y)\leq CQ(\tau)^{-\epsilon}.
\end{equation*}
\end{proof}

\subsection{Linearization and exponential convergence}$\ $

In the rest of this section, we fix the point $y\in \mathbb{H}^{n+1}$ obtained in Proposition \ref{s7:2-osc} and write $\Sigma_t$ as graphs of $u(\cdot,t)=u_y(\cdot,t)$ in the geodesic polar coordinate centered at $y$. We have proved in Proposition \ref{s7:2-osc} that the solution $\sigma=u-\theta$ of \eqref{eq-u-theta} converges to 0 exponentially. Next we show how to use the linearization to improve the convergence rate. This idea has been used earlier by Andrews in \cite{And99} to study the convergence of the affine curve-lengthening flow.

The linearized equation of flow \eqref{eq-u-theta} about the spherical solution is given as follows. If we have a family of solutions $\sigma(\xi,\tau,s):=u(\xi,\tau,s)-\theta(\tau)$ with $\sigma(\xi,\tau,0)=0$, then writing $\dot{\sigma}(\xi,\tau)=\F{\P}{\P s}\sigma(\xi,\tau,s)|_{s=0}$, we find by differentiating \eqref{s2:h-hat} and \eqref{eq-u-theta} the following equations:
\begin{equation*}
\left.\F{\P}{\P s}\hat{h}_i^j\right|_{s=0}=-\frac 1{\sinh^{2}\theta}\left(\dot{\sigma}\delta_i^j+\bar{\nabla}_i\bar{\nabla}^j\dot{\sigma}\right),
\end{equation*}
\begin{align*}
\F{\P}{\P \tau}\dot{\sigma}=&\frac{p}{F^{p+1}Q^p\sinh^2u}\biggr|_{u=\theta}\left(n\dot{\sigma}+\Delta_{\mathbb{S}^n}\dot{\sigma}\right)\nonumber\\
=&\frac{2p}{n^{p+1}(1-e^{-2\theta})}\left(n\dot{\sigma}+\Delta_{\mathbb{S}^n}\dot{\sigma}\right).
\end{align*}
Let $G[\sigma]$ denote the right hand side of \eqref{eq-u-theta} and define
\begin{equation}\label{s7:3-A}
  A(\theta)=\frac{2p}{n^{p+1}(1-e^{-2\theta})}.
\end{equation}
Since $\sigma(\cdot,\tau)$ converges to zero exponentially as $\tau\to\infty$ as shown in Proposition \ref{s7:2-osc}, the map $G$ is a smooth map from a sufficiently small $C^2(\mathbb{S}^n)$-neighborhood of $\sigma=0$ to $C^0(\mathbb{S}^n)$, and we can write
\begin{align*}
  G[\sigma]=&G[0]+DG|_0(\sigma)+\eta\\
  =&A(\theta)\left(n+\Delta_{\mathbb{S}^n}\right)\sigma+\eta
\end{align*}
for sufficiently large time $\tau$, where $\eta$ denotes the error term which satisfies $|\eta|\leq C|\sigma|^2_{C^2(\mathbb{S}^n)}$ and $C$ does not depend on the time $\tau$. To see this, one may calculate $\eta$ by taking $\sigma(\xi,\tau,s)=s\cdot \sigma(\xi,\tau)$. Then by Taylor expansion, $\eta=\frac{1}{2}\frac{d^2}{ds^2}\bigg|_{s=s_0}G[\sigma(s)]$ for some $s_0\in[0,1]$. And
\begin{equation}\label{err-linearize}
	\begin{aligned}
		\frac{d^2}{ds^2}G[\sigma(s)]=\frac{1}{F^pQ^p}&\left[-pvF^{-1}\frac{\partial^2 F}{\partial \hat{h}_i^j\partial\hat{h}_k^l}\cdot\frac{d}{ds}\hat{h}_i^j\cdot\frac{d}{ds}\hat{h}_k^l-pvF^{-1}\frac{\partial F}{\partial\hat{h}_i^j}\cdot\frac{d^2}{ds^2}\hat{h}_i^j\right.\\
		&\ +p(p+1)vF^{-2}\left(\frac{\partial F}{\partial\hat{h}_i^j}\cdot\frac{d}{ds}\hat{h}_i^j\right)^2+\frac{d^2}{ds^2}v\\
		&\ \left.-pF^{-1}\frac{\partial F}{\partial\hat{h}_i^j}\frac{d}{ds}\hat{h}_i^j\frac{d}{ds}v\right],\ \ \ \ s\in[0,1],
	\end{aligned}
\end{equation}
where $F=F(\hat{h}_i^j),\ \hat{h}_{i}^j$ are functions involving $\sigma(s)$ and its first, second derivatives. By \eqref{s2:h-hat}, \eqref{GN-ineq} and that $\lim_{\tau\to\infty}\sigma(\xi,\tau,s)=0$, one can calculate directly and show that $\sinh^2u=O(Q)$, $v=O(1)$, $\frac{d^\alpha}{ds^\alpha}v\leq CQ^{-1}|\sigma|_{C^2}^2, \alpha=1,2,$ and $\hat{h}_i^j=Q^{-1}\delta_i^j+o(Q^{-1})$, $\big|\frac{d}{ds}\hat{h}_i^j\big|^2\leq CQ^{-2}|\sigma|_{C^2}^2$, $\big|\frac{d^2}{ds^2}\hat{h}_i^j\big|\leq CQ^{-1}|\sigma|_{C^2}^2$. Then $F\geq CQ^{-1}>0$, $\frac{\partial F}{\partial\hat{h}_i^j}=O(1)$, $\frac{\partial^2 F}{\partial \hat{h}_i^j\partial\hat{h}_k^l}=O(Q)$. Combining all these together and substituting them into \eqref{err-linearize}, we get that $\eta$ is controlled by $|\sigma|_{C^2}^2$ uniformly.

We decompose $\sigma$ in the form
\begin{equation*}
  \sigma(\cdot,\tau)=\sum_{k}\phi_k(\cdot,\tau),
\end{equation*}
where each $\phi_k(\cdot,\tau)$ is a harmonic homogeneous polynomial on $\RR^{n+1}$ of degree $k$. We have
 \begin{equation*}
   \Delta_{\mathbb{S}^n}\phi_k=-k(n-1+k)\phi_k,\qquad k\geq 0.
 \end{equation*}
This gives that
\begin{align}
\F{d}{d\tau}\int_{\mathbb{S}^n}\sigma^2d\mu=&2\int_{\mathbb{S}^n}\sigma G[\sigma]d\mu\nonumber\\
=&2\int_{\mathbb{S}^n}\sigma\left(A(\theta)(n+\Delta_{\mathbb{S}^n})\sigma+\eta\right)d\mu\nonumber\\
=&2\int_{\mathbb{S}^n}\left(A(\theta)\sum_{k\geq 0}\left(n-k\left(n-1+k\right)\right)\phi_k^2+\sigma\eta\right)d\mu,\label{evo-sigma2}
\end{align}
and
\begin{align}
\F{d}{d\tau}\int_{\mathbb{S}^n}\sigma d\mu=&\int_{\mathbb{S}^n}\left(A(\theta)n\sigma +\eta\right)d\mu.\label{evo-sigma}
\end{align}
Since $\phi_1$ is a linear combination of the first eigenfunctions of $\Delta_{\mathbb{S}^n}$ which are given by the coordinate functions $z^i$, $i=1,\cdots,n+1$ of $\mathbb{S}^n$ in $\mathbb{R}^{n+1}$, we also have
\begin{align}\label{evo-sigmaz}
\F{d}{d\tau}\int_{\mathbb{S}^n}\sigma z^i d\mu=&\int_{\mathbb{S}^n}z^i\left(A(\theta)(n+\Delta_{\mathbb{S}^n})\sigma+\eta\right)d\mu\nonumber\\
=&\int_{\mathbb{S}^n}z^i\left(A(\theta)(n+\Delta_{\mathbb{S}^n})\phi_1+\eta\right)d\mu\nonumber\\
=&\int_{\mathbb{S}^n}\eta z^id\mu.
\end{align}
Note that we have
\begin{equation*}
  \phi_0=\frac{1}{\omega_n}\int_{\mathbb{S}^n}\sigma d\mu,~\phi_1=\frac{n+1}{\omega_n}\sum_{k=1}^{n+1}(\int_{\mathbb{S}^n}\sigma z^kd\mu)z^k,
\end{equation*}
where $\omega_n=|\mathbb{S}^n|$. By setting
\begin{equation*}
  \sigma_{0}:=\sigma-\phi_0,\qquad\sigma_{1}:=\sigma-\phi_0-\phi_1,
\end{equation*}
we have
\begin{equation}\label{evo-sigmai}
\F{d}{d\tau}\int_{\mathbb{S}^n}\sigma_i^2d\mu=2\int_{\mathbb{S}^n}\left(A(\theta)\sum_{k\geq i+1}\left(n-k\left(n-1+k\right)\right)\phi_k^2+\sigma_i\eta_i\right)d\mu
\end{equation}
for $i=0,1$, where each $\eta_i$ can be controlled by $|\sigma|^2_{C^2(\mathbb{S}^n)}$.

We have the following estimates:
\begin{lem}
Under the assumptions of Theorem \ref{thm-0<p<1}, there exists $\tau_0>0$ and some positive constants $\gamma,\ C$ such that for any $\tau>\tau_0$ we have
\begin{align}
&\int_{\mathbb{S}^n}\phi_0^2(\cdot,\tau)d\mu\leq  ~C\left(\int_{\mathbb{S}^n}\sigma^2(\cdot,\tau)d\mu\right)^{1+\gamma},\label{phi0-decay}\\
&\int_{\mathbb{S}^n}\phi_1^2(\cdot,\tau)d\mu\leq  ~C\left(\int_{\mathbb{S}^n}\sigma^2(\cdot,\tau)d\mu\right)^{1+\gamma}.\label{phi1-decay}
\end{align}
\end{lem}
\begin{proof}
By a special case of the Gagliardo-Nirenberg interpolation inequality, we have
\begin{equation}\label{GN-ineq}
|\sigma|_{C^2}\leq C(k)|\sigma|_{C^k}^{\frac{n+2}{n+k}}|\sigma|_{L^2}^{\frac{k-2}{n+k}}
\end{equation}
for any $k\geq n$. Then the $C^{\infty}$ estimate of $\sigma$ implies that for any $0<\epsilon<1$, we have
\begin{equation}\label{sigma-C2}
|\sigma|_{C^2}^2\leq C(\epsilon)|\sigma|_{L^2}^{1+\epsilon}.
\end{equation}
Note that
\begin{equation*}
  \int_{\mathbb{S}^n}\phi_0^2d\mu=\frac 1{\omega_n}\left(\int_{\mathbb{S}^n}\sigma d\mu\right)^2,
\end{equation*}
where $\omega_n=|\mathbb{S}^n|$. To prove the inequality \eqref{phi0-decay}, it suffices to prove that
\begin{equation*}
  G=\left(\int_{\mathbb{S}^n}\sigma d\mu\right)^2-\left(\int_{\mathbb{S}^n}\sigma_{0}^2d\mu\right)^{1+\gamma}
\end{equation*}
is non-positive when $\tau>\tau_0$ with $\tau_0$ sufficiently large. By \eqref{evo-sigma}, \eqref{evo-sigmai} and \eqref{sigma-C2}, when $\tau>\tau_0>>0$, we have
\begin{align*}
\F{d}{d\tau}G\geq&~2nA(\theta)\left(\int_{\mathbb{S}^n}\sigma d\mu\right)^2-2(1+\gamma)A(\theta)\left(\int_{\mathbb{S}^n}\sigma_{0}^2d\mu\right)^{\gamma}\int_{\mathbb{S}^n}\sum_{k\geq 1}\left(n-k(n-1+k)\right)\phi_k^2d\mu\nonumber\\
&-c\bigg|\int_{\mathbb{S}^n}\sigma d\mu\bigg |\int_{\mathbb{S}^n}|\sigma|_{C^2}^2d\mu-c(1+\gamma)\left(\int_{\mathbb{S}^n}\sigma_{0}^2d\mu\right)^{\gamma}\int_{\mathbb{S}^n}|\sigma_0||\sigma|_{C^2}^2d\mu\nonumber\\
\geq&~2nA(\theta)\left(\int_{\mathbb{S}^n}\sigma d\mu\right)^2-c|\sigma|_{L^2}^{2+\epsilon}-c(1+\gamma)|\sigma|_{L^2}^{2+\epsilon+2\gamma}\nonumber\\
=&~nA(\theta)\left(G+\left(\int_{\mathbb{S}^n}\sigma_{0}^2d\mu\right)^{1+\gamma}\right)+n\omega_nA(\theta)\int_{\mathbb{S}^n}\phi_0^2 d\mu\nonumber\\
&~-c|\sigma|_{L^2}^{2+\epsilon}-c(1+\gamma)|\sigma|_{L^2}^{2+\epsilon+2\gamma}\nonumber\\
\geq&~nA(\theta)G
\end{align*}
if we choose $\epsilon>2\gamma$ with $\gamma<\frac{1}{2}$, where the last inequality above is due to the fact that
\begin{equation*}
  \int_{\mathbb{S}^n}\sigma^2d\mu=\int_{\mathbb{S}^n}\phi_0^2d\mu+\int_{\mathbb{S}^n}\sigma_{0}^2d\mu
\end{equation*}
converges to $0$ by Proposition \ref{s7:2-osc}. Since $A(\theta)$ converges to $\frac{2p}{n^{p+1}}$ as $\tau$ goes to infinity, $G$ must be non-positive when $\tau>\tau_0>>0$, otherwise $G$ would become unbounded which contradicts Proposition \ref{s7:2-osc}.

Similarly, noting that
\begin{equation*}
  \int_{\mathbb{S}^n}\phi_1^2d\mu=\frac{n+1}{\omega_n}\sum_{i=1}^{n+1}\left(\int_{\mathbb{S}^n}\sigma z^id\mu\right)^2,
\end{equation*}
to prove the inequality \eqref{phi1-decay} it suffices to prove that
\begin{equation*}
  \tilde{G}=\sum_{i=1}^{n+1}\left(\int_{\mathbb{S}^n}\sigma z^id\mu\right)^2-\left(\int_{\mathbb{S}^n}\sigma_{1}^2d\mu\right)^{1+\gamma}
\end{equation*}
stays non-positive when $\tau>\tau_0$ with $\tau_0$ sufficiently large. Again, by \eqref{evo-sigmaz}, \eqref{evo-sigmai} and \eqref{sigma-C2}, when $\tau>\tau_0>>0$, we have
\begin{align}\label{s7:3-3}
\F{d}{d\tau}\tilde{G}\geq&-2(1+\gamma)A(\theta)\left(\int_{\mathbb{S}^n}\sigma_{1}^2d\mu\right)^{\gamma}\int_{\mathbb{S}^n}\sum_{k\geq 2}\left(n-k(n-1+k)\right)\phi_k^2d\mu\nonumber\\
&-c\sum_{i}\left|\int_{\mathbb{S}^n}\sigma z^id\mu\right|\int_{\mathbb{S}^n}|\sigma|_{C^2}^2d\mu-c(1+\gamma)\left(\int_{\mathbb{S}^n}\sigma_{1}^2d\mu\right)^{\gamma}\int_{\mathbb{S}^n}|\sigma_1||\sigma|_{C^2}^{2}d\mu\nonumber\\
\geq&(2-\alpha)(1+\gamma)A(\theta)(n+2)|\sigma_{1}|_{L^2}^{2+2\gamma}\nonumber\\
&+\alpha(1+\gamma)A(\theta)(n+2)\left(-\tilde{G}+\sum_{i=1}^{n+1}\left(\int_{\mathbb{S}^n}\sigma z^id\mu\right)^2\right)\nonumber\\
&-c|\sigma_{0}|_{L^2}^{2+\epsilon}-c(1+\gamma)|\sigma_{0}|_{L^2}^{2+\epsilon+2\gamma}\nonumber\\
=&-\alpha(1+\gamma)A(\theta)(n+2) \tilde{G}\nonumber\\
&+(1+\gamma)A(\theta)(n+2)\left((2-\alpha)|\sigma_{1}|_{L^2}^{2+2\gamma}+\frac{\alpha\omega_n}{n+1}\int_{\mathbb{S}^n}\phi_1^2d\mu\right)\nonumber\\
&-c|\sigma_{0}|_{L^2}^{2+\epsilon}-c(1+\gamma)|\sigma_{0}|_{L^2}^{2+\epsilon+2\gamma}
\end{align}
where $\alpha\in (0,2)$ is a constant to be determined and we have used the fact that $|\sigma|_{L^2}\leq C|\sigma_{0}|_{L^2}$ due to the inequality \eqref{phi0-decay}. We choose $\alpha$ such that $\alpha<2$ and
\begin{equation*}
  \alpha\leq\frac{n\delta}{6p(n+2)(1+\gamma)},
\end{equation*}
where $\delta$ is the constant determined in Lemma \ref{est-umbilic}, and choose $\epsilon>2\gamma>0$.  Using the relation
\begin{equation*}
  \int_{\mathbb{S}^n}\sigma_0^2d\mu=\int_{\mathbb{S}^n}\sigma_1^2d\mu+\int_{\mathbb{S}^n}\phi_1^2d\mu,
\end{equation*}
we obtain that the last two lines of \eqref{s7:3-3} are nonnegative for $\tau> \tau_0$ with $\tau_0$ sufficiently large. Therefore
\begin{equation*}
\F{d}{d\tau}\tilde{G}\geq-\frac{n\delta}{6p}A(\theta) \tilde{G}
\end{equation*}
for $\tau>\tau_0$ with $\tau_0$ sufficiently large. By the expression \eqref{s7:3-A} of $A(\theta)$, for $\tau_0$ sufficiently large we have $A(\theta)\leq 3pn^{-p-1}$ for any $\tau>\tau_0>>0$. Then
\begin{equation}\label{s7:2-Gtld2}
\F{d}{d\tau}\tilde{G}\geq-\frac{\delta}{2n^p}\tilde{G}
\end{equation}
for $\tau> \tau_0$ with $\tau_0$ sufficiently large. If there exists a sufficiently large time $\tau_1>\tau_0$ such that $\tilde{G}$ is positive at $\tau_1$. The inequality \eqref{s7:2-Gtld2} implies that $\tilde{G}(\tau)\geq \tilde{G}(\tau_0) e^{-\frac{\delta}{2n^p}(\tau-\tau_0)}$. However, Proposition \ref{s7:2-osc} implies that $\tilde{G}(\tau)\leq Ce^{-\frac{2\epsilon}{n^p}\tau}$ for any $\epsilon<\delta/2$. This is a contradiction. Therefore we conclude that $\tilde{G}$ remains non-positive for all $\tau> \tau_0$ with $\tau_0$ sufficiently large.
\end{proof}

\begin{prop}\label{linear-decay}
Under the assumptions of Theorem \ref{thm-0<p<1}, for any $\beta<(1+\frac{2}{n})p$, there exists a positive constant $C=C(p,\beta,\Sigma_0)$ such that
\begin{equation}\label{s7:3-1}
|u-\theta|\leq CQ(t)^{-\beta}.
\end{equation}
\end{prop}
\begin{proof}
We rewrite the equation \eqref{evo-sigma2} as
\begin{align*}
  \F{d}{d\tau}\int_{\mathbb{S}^n}\sigma^2d\mu =& -2(n+2)A(\theta)\int_{\mathbb{S}^n}\sum_{k\geq 0}
\phi_k^2d\mu \\
   &\quad +(4n+4) A(\theta)\int_{\mathbb{S}^n}\phi_0^2d\mu+2(n+2)A(\theta)\int_{\mathbb{S}^n}\phi_1^2d\mu\\
   &\quad +2\int_{\mathbb{S}^n}\left(-\sum_{k\geq 2}(n(k-2)+k^2-k-2)\phi_k^2+\sigma\eta\right)d\mu\\
   \leq &-2(n+2)A(\theta)\int_{\mathbb{S}^n}\sigma^2d\mu  +(4n+4) A(\theta)\int_{\mathbb{S}^n}\phi_0^2d\mu\\
   &\quad +2(n+2)A(\theta)\int_{\mathbb{S}^n}\phi_1^2d\mu+C\int_{\mathbb{S}^n}|\sigma||\sigma|^2_{C^2}d\mu.
\end{align*}
Applying the estimates \eqref{phi0-decay}, \eqref{phi1-decay} and \eqref{sigma-C2}, we have
\begin{align*}
\F{d}{d\tau}\int_{\mathbb{S}^n}\sigma^2d\mu\leq&-2(n+2)A(\theta)\int_{\mathbb{S}^n}\sigma^2d\mu+C|\sigma|_{L^2}^{2+\epsilon}+C|\sigma|_{L^2}^{2+2\gamma}.
\end{align*}
Since $2(n+2)A(\theta)=\F{4(n+2)p}{n^{p+1}(1-e^{-2\theta})}$ converges to $4(1+\F{2}{n})n^{-p}p$ increasingly, for any $\beta<(1+\frac{2}{n})p$ there exists a constant $C=C(p,\beta,\Sigma_0)$ with
$$|u-\theta|_{L^2}\leq Ce^{-\frac 2{n^p}\beta\tau}.$$
By \eqref{s5:thet-tau} and the definition of $Q(t)=Q(t(\tau))$, we conclude that for any $\beta<(1+\frac{2}{n})p$, we have
\begin{equation*}
  |u-\theta|_{L^2}\leq CQ(t)^{-\beta}
\end{equation*}
for some positive constant $C(p,\beta,\Sigma_0)$. Then \eqref{s7:3-1} follows from the interpolation inequality \eqref{GN-ineq}.
\end{proof}

\begin{cor}
We have the following improved estimate on shifted curvature:
\begin{equation}
|\kappa_iQ-1|\leq CQ^{-\beta},\qquad \forall\ \beta<(1+\frac{2}{n})p.
\end{equation}
\end{cor}
\begin{proof}
For any $\beta<(1+\frac{2}{n})p$, by Proposition \ref{linear-decay} we have
\begin{equation*}
\text{osc}(u)\leq CQ^{-\beta}.
\end{equation*}
The proof of Lemma \ref{kappaQ=1} implies that
\begin{equation*}
  |\kappa_iQ-1|\leq CQ^{-\F{\beta}{2}}.
\end{equation*}
Then
\begin{align}\label{s7:3-2}
h_i^j-\frac{\coth(u)}v\delta_i^j=&\hat{h}_i^j+(1-\frac{\coth u}{v})\delta_i^j~\geq~-CQ^{-1-\beta/2}\delta_i^j.
\end{align}
Substituting \eqref{s7:3-2} into \eqref{s4:G3} and comparing with \eqref{s4:v2}, we have
\begin{align*}
0\leq&\left(k+CvQ^{-1-\beta/2}\right)v^2|\bar{\nabla}u|^2
\end{align*}
at the maximum point of the function $G=\ln v+ku$. By choosing $k=-2CQ^{-1-\beta/2}$ and applying the same argument in Lemma \ref{est-C1}, we obtain
\begin{align*}
v\leq&\text{exp}(|k|\text{osc}(u))~\leq~\text{exp}(CQ^{-1-\frac 32\beta}).
\end{align*}
Hence a similar argument in Lemma \ref{kappaQ=1} gives that
\begin{equation*}
|\kappa_iQ-1|\leq CQ^{-\epsilon},\quad \forall\ \epsilon<\F{3}{4}\beta.
\end{equation*}
Repeating the argument above we can obtain that $|\kappa_iQ-1|\leq C(\epsilon,p,\Sigma_0)Q^{-\epsilon}$ holds for all $\epsilon<\F{2^m-1}{2^m}\beta,\ m\in\mathbb{Z}^+$. Thus by Lemma \ref{kappaQ=1}, our assertion holds.
\end{proof}

Finally we show that the flow actually becomes arbitrarily close to a flow of geodesic spheres. Since the radius $\theta(t)=\theta(t,T^*)$ of the spherical solution $S_t$ is chosen such that $T^*$ is the same maximal existence time of flow \eqref{eq-flow}, then we  have
\begin{equation}
S_t\cap\Sigma_{t}\neq\emptyset
\end{equation}
for all $t\in [0,T^*)$.  Then Theorem \ref{thm-0<p<1} follows from Proposition \ref{linear-decay}.

\section{Example: loss of horo-convexity}\label{sec:examp}

In this section, we give an example of a horo-convex hypersurface in $\HH^{n+1}$ which develops a principal curvature less than 1 instantly along the flow
\begin{equation}\label{s8:IMCF}
\F{\P}{\P t}X=\frac{1}{\hat{H}^p}\vec{\nu},\ \ p>1,
\end{equation}
in $\mathbb{H}^3$, where $\hat{H}=H-n$ is the shifted mean curvature. By the continuity with respect to initial values this shows that one can also find strictly horo-convex initial hypersurface which develop a principal curvature less than 1 quickly along the flow \eqref{s8:IMCF}.  This indicates that the condition $f|_{\P\Gamma^+}=0$ in Theorem \ref{thm-0<p<infty} (when $p>1$) is needed, and that we cannot generalize our results in Theorem \ref{thm-0<p<1} (cases (b) (c) (d)) for $p>1$.

We use the upper half-space model of $\HH^3$: the hyperbolic space $\HH^3$ is the upper half-space $\RR^3_+$ equipped with a conformally flat metric
\begin{equation*}
\bar{g}(x)=\frac{dx_1^2+dx_2^2+dx_3^2}{x_3^2},\ \ x\in\RR^3_+.
\end{equation*}
We distinguish quantities in $\RR^3_+$ from those in $\HH^{3}$ by an additional br\`eve.
We recall the construction given in \cite[\S 5]{A-M-Z} on a smooth convex hypersurface in Euclidean space which losses convexity along the contracting curvature flow (see \cite{K-J} for another application of this construction). The idea is to replace a small portion of a sphere by local graph of a smooth strictly convex function. Let $R_0>R$ be two positive constants. We consider the the following bounded region
\begin{equation*}
  \Omega=\left\{\mathbf{x}=(x_1,x_2,x_3):~ \breve{u}(x)\leq x_3\leq R_0+\sqrt{R^2-|x|^2},~|x|\leq R\right\}
\end{equation*}
in $\mathbb{R}^3_+$, where we denote $x=(x_1,x_2)\in \mathbb{R}^2$, $\breve{u}$ is a smooth positive function to be determined later with $\breve{u}(x)=R_0-\sqrt{R^2-|x|^2}$ outside a ball $B_r(0)\in \mathbb{R}^2$ for sufficiently small $r>0$. We aim to show that for suitable function $\breve{u}$, the boundary of $\Omega$ provides an example of horo-convex hypersurface in hyperbolic space which losses horo-convexity along the flow \eqref{s8:IMCF}.

We modify the function $\breve{u}$ used in \cite[\S 5]{A-M-Z} by adding a positive constant $c_3$. That is, we consider the graph $\Sigma_1=\{\left(x_1,x_2,\breve{u}(x_1,x_2)\right)\mid(x_1,x_2)\in\RR^2\}$ with the function $\breve{u}$ given by
\begin{equation}
\breve{u}(x)=\frac{c_1}{24}x_1^4+\frac{1}{2}(a_2+b_2x_1+\frac{1}{2}c_2x_1^2)x_2^2+c_3,
\end{equation}
where $a_2,\ b_2,\ c_3$ are arbitrary positive constants and
\begin{equation*}
c_1=\frac{1}{4},\ \ c_2=\frac{2b_2^2}{a_2}+\frac{1}{4}.
\end{equation*}
In this coordinate, the induced metric $g_{ij}$ on $\Sigma_1$ is expressed as
\begin{equation*}
  g_{ij}=\frac{\check{u}_{,i}\check{u}_{,j}+\delta_{ij}}{\check{u}^2}.
\end{equation*}
The second fundamental form $h_{ij}$ of $\Sigma_1$ in  hyperbolic space is related to the second fundamental form $\breve{h}_{ij}$ with respect to the Euclidean metric by the following identity (see e.g., \cite[\S 1]{GS11})
\begin{align}\label{s8:h-bre}
h_{ij}=&\frac{\breve{h}_{ij}}{\breve{u}}+\frac{1}{v}g_{ij}~=~\frac{\breve{u}_{,ij}}{v\breve{u}}+\frac 1{v\breve{u}^2}(\breve{u}_{,i}\breve{u}_{,j}+\delta_{ij}),
\end{align}
where $v=\sqrt{1+|D\breve{u}|^2}$ and indices appearing after a comma denote usual partial derivatives.

At the point $(0,0)$, we have $\breve{u}(0)=c_3$ and $D\breve{u}(0)=0$. Then $g_{ij}(0)=c_3^{-2}\delta_{ij}$, the Christoffel symbols of the induced metric $g_{ij}$ satisfy $\Gamma_{ij}^k(0)=0$ and
\begin{align*}
  \partial_m\Gamma_{ij}^k(0)=&\check{u}_{,km}\check{u}_{,ij}-\frac 1{c_3}(\delta_{kj}\breve{u}_{,im}+\delta_{ik}\breve{u}_{,jm}-\delta_{ij}\breve{u}_{,km}).
\end{align*}
The second fundamental form satisfies
\begin{align*}
h_{ij}=&\frac 1{c_3}\breve{u},_{ij}+\frac 1{c_3^2}\delta_{ij},\qquad \quad \nabla_kh_{ij}=\frac 1{c_3}\breve{u},_{ijk},\\
\nabla_l\nabla_kh_{ij}=&\frac 1{c_3}\left(\breve{u},_{ijkl}-\breve{u},_{ki}\breve{u},_{jm}\breve{u},_{ml}-\breve{u},_{kj}\breve{u},_{im}\breve{u},_{ml}-\breve{u},_{ij}\breve{u},_{km}\breve{u},_{ml}\right)\\
&-\frac 1{c_3^2}\left(\breve{u},_{mk}\breve{u},_{ml}\delta_{ij}+\breve{u}_{lm}\breve{u}_{,mj}\delta_{ik}+\breve{u}_{,lm}\breve{u}_{,im}\delta_{jk}\right)\\
&+\frac 1{c_3^2}\left(\breve{u},_{kl}\breve{u},_{ij}+\breve{u}_{,kj}\breve{u}_{,il}+\breve{u}_{,lj}\breve{u}_{,ik}\right),
\end{align*}
at the point $(0,0)$. One can compute that at the point $(0,0)$, we have
\begin{align*}
&\hat{h}_{11}=h_{11}-g_{11}=0,\quad  \hat{h}_{22}=h_{22}-g_{22}=\frac{a_2}{c_3},\quad  \hat{h}_{12}=\hat{h}_{21}=0,\\
&\nabla_1h_{11}=0,\qquad \nabla_1h_{22}=\frac{b_2}{c_3},\\
&\nabla_1\nabla_1h_{11}=\frac{c_1}{c_3},\qquad  \nabla_2\nabla_2h_{11}=\frac{c_2}{c_3}-\frac{a_2^2}{c_3^2}.
\end{align*}
By the equation \eqref{s3:1-ssff}, the entry $\hat{h}_{11}(0,t)$ of the shifted second fundamental form evolves along the flow \eqref{s8:IMCF} by
\begin{align}\label{s8:dh11}
\frac{\P}{\P t}\hat{h}_{11}=&\frac{pc_3^2}{\hat{H}^{p+1}}(\nabla_1\nabla_1h_{11}+\nabla_2\nabla_2h_{11})-\frac{p(p+1)c_3^4}{\hat{H}^{p+2}}(\nabla_1h_{11}+\nabla_1h_{22})^2\nonumber\\
&\quad +\frac{p}{\hat{H}^{p+1}}\left((\hat{h}_1^1)^2+(\hat{h}_2^2)^2\right)\frac 1{c_3^2}\nonumber\\
=& \frac p{(c_3a_2)^{p+1}}\left(c_3(c_1+c_2)-a_2^2\right)-\frac{p(p+1)}{(c_3a_2)^{p+2}}(c_3b_2)^2+\frac p{(c_3a_2)^{p+1}}a_2^2\nonumber\\
=&\frac{p}{c_3^pa_2^{p+2}}\left(\frac{a_2}{2}+(1-p)b_2^2\right).
\end{align}
Since $p>1$, when $c_3$ is fixed, for suitable $a_2$ and $b_2$, the entry $\hat{h}_{11}(0,t)$ satisfies $\frac{\P}{\P t}\hat{h}_{11}<0$ at the point $(x_1,x_2)=(0,0)$. Hence $\hat{h}_{11}$ drops below than $0$ instantly at the point $(x_1,x_2)=(0,0)$.

Next, we show that the graph $\Sigma_1$ of $\breve{u}$  over a small ball $B_r(0)$ with the induced metric from the hyperbolic metric is a horo-convex hypersurface, and is strictly horo-convex at $(x_1,x_2)\neq (0,0)$: By direct computation, we have
\begin{align}\label{local1}
|D\breve{u}|^2=&a_2^2x_2^2+o(x_1^2+x_2^2),\\
\breve{h}_1^1=&\frac{c_1}{2}x_1^2+\frac{c_2}{2}x_2^2+o(x_1^2+x_2^2),\nonumber\\
\breve{h}_1^2=\breve{h}_2^1=&b_2x_2+c_2x_1x_2+o(x_1^2+x_2^2),\nonumber\\
\breve{h}_2^2=&a_2+b_2x_1+\frac{c_2}{2}x_1^2-a_2^3x_2^2+o(x_1^2+x_2^2),\nonumber
\end{align}
where the notation little $o$ means arbitrarily small comparing the function $x_1^2+x_2^2$. The identity \eqref{s8:h-bre} is equivalent to
\begin{equation}\label{s8:h-bre2}
  h_i^j=\breve{h}_i^j\breve{u}+\frac 1v\delta_i^j.
\end{equation}
Therefore the graph $\Sigma_1$ is strictly horo-convex if and only if
\begin{equation}\label{local-horo}
\breve{h}_i^j\breve{u}v~>~(v-1)\delta_{i}^j=\left(\frac{1}{2}|D\breve{u}|^2+o(|D\breve{u}|^2)\right)\delta_{i}^j.
\end{equation}
One can compute that the smallest eigenvalue of $(\breve{h}_i^j\breve{u})$ equals
\begin{equation}\label{local2}
\lambda_1(\breve{h}_i^j\breve{u}v)=\frac{c_1c_3}{2}x_1^2+c_3(\frac{1}{2}c_2-\frac{b_2^2}{a_2})x_2^2+o(x_1^2+x_2^2)=\frac{c_3}{8}x_1^2+\frac{c_3}{8}x_2^2+o(x_1^2+x_2^2).
\end{equation}
Then by \eqref{local1}, \eqref{local2} and continuity, it is not difficult to see that the graph of $\breve{u}$ is horo-convex (i.e., \eqref{local-horo} is satisfied) in a small ball $B_r(0)$ if $c_3>4a_2^2$.

By the similar argument as given in \cite[\S 4]{A-M-Z}, we can modify the function $\breve{u}$ outside $B_{r/4}(0)$, keeping it uniformly horo-convex and smooth, to make it equal to $R_0-\sqrt{R^2-|x|^2}$ outside $B_r(0)$ for suitable $R$ and $R_0$. We denote by $\Sigma_2$ the part $\partial \Omega\setminus (\Sigma_1|_{B_{r/4}(0)})$ which is strictly convex in $\RR^3_+$ with respect to the Euclidean metric. Suppose that $\breve{h}_i^j\geq\beta\delta_i^j$ on $\Sigma_2$, where $\beta=\beta(\Sigma_2)>0$ is a constant which is invariant under translation in $\RR^3_+$. Then by choosing $c_3>0$ sufficiently large and applying \eqref{s8:h-bre2}, we have
\begin{equation}\label{s8:h-2}
h_i^j\geq \beta c_3\delta_i^j+\frac{1}{v}\delta_i^j>\delta_i^j
\end{equation}
on $\Sigma_2$ which means that $\Sigma_2$ is horo-convex with respect to the hyperbolic metric. Meanwhile, when we enlarge $c_3$, the principal curvatures of $\Sigma_1$ in $\mathbb{H}^3$ will not decrease. Hence, $\partial\Omega$ provides a smooth closed horo-convex hypersurface which losses the horo-convexity instantly along flow \eqref{s8:IMCF}.

\medskip
\noindent\textbf{Acknowledgments.} The authors would like to thank Professor Ben Andrews for helpful discussions, especially for his suggestion on the application of linearization to improve the convergence rate of the flow in Theorem \ref{thm-0<p<1}. The first author and the third author are also grateful to the Mathematical Sciences Institute at the Australian National University for its hospitality during their visit, when part of this work was completed. The authors would also like to thank the referee for his/her
 carefull reading of the manuscript and the valuable comments and suggestions. The research was supported by the National Key Research and Development Project (2020YFA0713100 and 2021YFA1001800), the National Natural Science Foundation of China (Grant No.11971244), the Fundamental Research Funds for the Central Universities from Nankai University, and Research grant KY0010000052 from University of Science and Technology of China.

\medskip
\noindent\textbf{Data availability}: Data sharing not applicable to this article as no datasets were generated or analysed during the current study.


\end{document}